\documentclass{article}
\usepackage{geometry}
\usepackage{tikz}
\usetikzlibrary{arrows.meta, positioning}

 \usepackage{amsfonts,amsmath,amsthm,amstext,amssymb,url}
 
\usepackage{multirow}
\usepackage{longtable}
\usepackage{float}
 
\usepackage{pgfplots}
\usepackage{cancel}
\usepackage{hyperref}

\newtheorem{thm}{Theorem}[section]
\newtheorem{defin}[thm]{Definition}
\newtheorem{lem}[thm]{Lemma}
\newtheorem{cor}[thm]{Corollary}
\newtheorem{rem}[thm]{Remark}
\newtheorem{examp}[thm]{Example}
\newtheorem{propst}[thm]{Proposition}

\usepackage{multicol}
 
\usepackage{amsmath}
\usepackage{color}
\usepackage[font=small,labelfont=bf]{caption}

\def\Q{\mathbb{Q}}
\def\R{\mathbb{R}}
\def\C{\mathbb{C}}
\def\Cr{{\cal C}_ {\mathbb R}}

\usepackage{soul}
\usepackage{authblk}

\title{Tools for analyzing the intersection curve between a torus and a quadric through projection and lifting}
\author[1]{Laureano Gonzalez-Vega}
\author[2]{Jorge Caravantes}
\author[3]{Gema M. Diaz-Toca}
\author[4]{Mario Fioravanti}

\affil[1]{CUNEF, Spain, \href{mailto:laureano.gonzalez@cunef.edu}{laureano.gonzalez@cunef.edu}}
\affil[2]{Universidad de Alcal\'a, Spain, \href{mailto:jorgecaravan@gmail.com}{jorgecaravan@gmail.com}}
\affil[3]{Universidad de Murcia, Spain, \href{mailto:gemadiaz@um.es}{gemadiaz@um.es}}
\affil[4]{Universidad de Cantabria, Spain, \href{mailto:mario.fioravanti@unican.es}{mario.fioravanti@unican.es}}

\begin{document}
\maketitle
\begin{abstract} 
This article introduces efficient and user-friendly tools for analyzing the intersection curve between a ringed torus and an irreducible quadric surface. Without loose of generality, it is assumed that the torus is centered at the origin, and its axis of revolution coincides with the $z$-axis. The paper primarily focuses on examining the curve's projection onto the plane $z=0$, referred to as the cutcurve, which is essential for ensuring accurate lifting procedures. Additionally, we provide a detailed characterization of the singularities in both the projection and the intersection curve, as well as the existence of double tangents. A key tool for the analysis is the theory of resultant and subresultant polynomials.
\end{abstract}

\section*{Introduction}
%

The intersection of two surfaces is a rich topic in algebraic geometry, which appears in many problems of Computer-Aided Design, Computer Aided Geometric Design, Computer animation, etc. In general, the intersection of two algebraic surfaces is itself a curve, and this curve has important geometric and topological properties that depend on the properties of the surfaces involved. In recent years, several papers have emerged addressing two main aspects: the study of the intersection curve for two specific surfaces (for example, see for two quadrics \cite{Berberich:2005:ECE:1064092.1064110}, \cite{DUPONT2008168}, \cite{Dupont2008c}, \cite{Dupont2008c3}, \cite{Gonzalez-Vega2021}, \cite{Schomer:2006}, \cite{Falai:2023}; for two tori \cite{Jia:2013}, for torus and sphere \cite{Kim1998}, for torus and ellipsoid \cite{breda2022}, etc); and the topology of a space curve by analyzing first the projection onto one plane and then, determining the lifting of the projection curve (see for example \cite{ALCAZAR2005719}, \cite{Pouget2024}, \cite{DMR}).

In this article, we analyze the intersection in $\mathbb{R}^3$ of a ring torus $\mathcal{T}$ centered at the origin and a real quadric $\mathcal{Q}$ by also studying the curve's projection onto the plane $z=0$, known as the cutcurve. However, unlike most jobs which study the topology of a curve, we do not require the intersection curve to be in pseudo general position. Recall that a space curve is said to be in pseudo generic position with respect to the $(x,y)$ plane if the projection $\Pi_z$ is injective up to a finite number of exceptions. This doesn't have to be true when studying the intersection of a torus and a quadric, as can be seen in several examples of Section \ref{secej}. Nor do we require the projection to be in general position. Neither of these two conditions must be met in our study. In this work, we focus on the study of the intersection curve by itself. For this, we will use tools of symbolic calculation, such as resultant and subresultants, and results of algebraic geometry.

We must mention that the intersection of a torus and a quadric have been previously considered under different points of view. For example, in \cite{bottema}, all real algebraic curves lying on a ring torus of degree less or equal to 6 are examined; 
the intersection of a torus and a quadric was also considered by \cite{lebesgue}, characterizing the cases in which it degenerates into two quartic  curves;  
in \cite{breda2024}, only 
the intersection curve of an hyperbolic cylinder with a torus sharing the same center is considered;  in \cite{breda2021}, it is introduded an implementation in GeoGebra of algorithms for visualizing the intersection curve of a quadric with a torus; and finally, in \cite{Kim}, \cite{Li2004AlgebraicAF}, the intersection with natural quadrics is considered.

The paper is organized as follows. Section \ref{sec1} recalls the well known notion of resultant, the subresultant polynomials, and we shall apply these concepts to the matter at hand, specifically that of a torus and a quadric. In Section \ref{sec2}, the definition of the cutcurve with its characterization is introduced. In Section \ref{lifting}, we study the relationship between the resultant, the first subresultant and the cutcurve. In Section \ref{lifting1}, we explore the relationship between the singularities of the cutcurve and those of the intersection curve. In Section \ref{ecuaciones}
 we present the expression of the cutcurve in terms of the coefficients of the implicit equation of the quadric, its relation with the silhouette curves of the torus and the quadric, and a characterization of the cutcurve singularities. Section \ref{secej} is dedicated to presenting examples which illustrate the results presented in the paper.

Through the paper, we will assume that the quadric is irreducible, i.e. neither the square of a plane nor the union of two planes.

\section{Resultants and subresultants}\label{sec1}

Resultants and subresultants are the algebraic tools used to determine both the projection of the intersection curve between the considered surfaces and its lifting from the plane to the 3D space, because they provide a very easy and compact way of characterizing the greatest common divisor of two polynomials when they involve parameters.

\medskip

The concept of polynomial determinant associated to a matrix provides one
of the usual ways to define Subresultant polynomials. Let $\Delta$ be a
$m\times n$ matrix with $m\leq n$. The determinant polynomial of $\Delta$,
$\hbox{\bf detpol}(\Delta)$, is defined as:
\[
\hbox{\bf detpol}(\Delta)=\sum_{k=0}^{n-m}\det(\Delta_k)x^{n-m-k}
\]
where $\Delta_k$ is the square submatrix of $\Delta$ consisting of the
first $m-1$ columns and the $(k+m)$--th column.

\begin{defin}\label{subres}
Let $$A(x)=\sum_{i=0}^m a_{i}x^i\qquad \hbox{and} \qquad B(x)=\sum_{i=0}^n b_{i}x^i$$ be two polynomials with coefficients in a field ($\Q$ or $\R$ in our case). The $i$--th Subresultant polynomial of $A$ and $B$, denoted by $
{Sres}_i(A,B)$, is defined as the determinant polynomial of
the following submatrix of Sylvester matrix of $A$ and $B$:
\[\overbrace{  \begin{pmatrix}
                       \begin{matrix}a_m&\ldots&a_0\\
                                  &\ddots&&\ddots& \\
                                  &&a_m&\ldots&a_0\\\end{matrix} \\
                       \begin{matrix}b_n&\ldots&b_0 \\
                                  &\ddots&&\ddots& \\
                                  &&b_n&\ldots&b_0\\\end{matrix} \\
\end{pmatrix}}^{n+m-i}
                  \begin{matrix} \left.\begin{matrix}\\ \\
\\\end{matrix}\right\}&n-i\\
                                 \left.\begin{matrix}\\ \\
\\\end{matrix}\right\}&m-i\\
                  \end{matrix}
                  \]
and we define the $i$--th subresultant coefficient of $A$ and $B$ with respect to $x$,
$ {sres}_i(A,B;x)$, as the coefficient of $x^i$ in $ Sres_i(A,B;x)$. Moreover, $sres_{i,j}$ denotes the coefficient of $x^j$ in the polynomial $Sres_i(A,B;x)$ for $j<i$. 
Observe that the resultant of $A$ and $B$ with respect to $x$ is $Sres_0(A,B;x)= sres_0(A,B;x).$

\end{defin}

There are many ways of defining and computing subresultants: for a short introduction, see \cite{Gonzalez-Vega2009} and the references cited therein. Subresultants allow an easy characterization of the degree of the greatest common divisor of two univariate polynomials whose coefficients depend on one or several parameters. More generally, the determinants $  sres_{i}(A,B;x)$, which are the
formal leading coefficients of the subresultant sequence for $A$
and $B$, can be used to compute the greatest common divisor
of $A$ and $B$, owing to the following equivalence:
\begin{equation}\label{gcd} {  Sres}_i(A,B;x)=\gcd(A,B)\Longleftrightarrow
\begin{cases}{  sres}_0(A,B;x)=\ldots={  sres}_{i-1}(A,B;x)=0& \cr
            \hfill{  sres}_i(A,B;x)\neq 0\hfill&
\cr\end{cases}\end{equation}

\bigskip
Suppose now that the torus and the quadric are defined respectively as follows,
$$
\mathcal{T}=\{(x,y,z)\in \mathbb{R}^3  :T(x,y,z)=0\} \text{ and }
\mathcal{Q}=\{(x,y,z)\in \mathbb{R}^3  :Q(x,y,z)= 0\},
$$
with 
$T(x,y,z):= z^4+p_2(x,y)z^2+p_0(x,y)\in \mathbb{R}[x,y,z]$,
$$
p_2(x,y) = 2(R^2 - r^2 + x^2 + y^2), $$
$$  p_0(x,y)=(R^2 + 2Rr + r^2 - x^2 - y^2)(R^2 - 2Rr + r^2 - x^2 - y^2), 0<r<R,
$$
and
$Q(x,y,z):=  z^2+q_1(x,y)z+q_0(x,y)\in \mathbb{R}[x,y,z]$,
$$q_1=ex + fy + i \;\text{ and } \; q_0=ax^2 + by^2 + dxy + gx + hy + j.$$
Then, regarding $T$ and $Q$ as polynomials in $z$ whose coefficients are polynomials in $x$ and $y$, we first observe that the resultant is not identically zero because the polynomials $T$ and $Q$ do not have a common factor, 
%

\begin{eqnarray}
{{Sres}_0}(x,y)&=&\left|
\begin {array}{cccccc} 1&0&{\it p_2}&0&{\it p_0}&0
\\ \noalign{\medskip}0&1&0&{\it p_2}&0&{\it p_0}
\\ \noalign{\medskip}1&{\it q_1}&{\it q_0}&0&0&0
\\ \noalign{\medskip}0&1&{\it q_1}&{\it q_0}&0&0
\\ \noalign{\medskip}0&0&1&{\it q_1}&{\it q_0}&0
\\ \noalign{\medskip}0&0&0&1&{\it q_1}&{\it q_0}
\end {array}\right| \label{S_0}
\\ 
&=&(  q_0q_1^2+ p_2q_0 -q_0^2    - p_0  )^2 + q_1^2(q_1^2 + p_2 - 2q_0)(p_0-q_0^2 ),\label{s0eq}
\end{eqnarray}
and the total degree of $ {  Sres}_0(x,y)$ in $x$ and $y$ is at most eight. Moreover, if $z_1$ and $z_2$ are the symbolic roots of $Q$,
$$
z_1={\frac {1}{2 } \left( - q_1+\sqrt {q_1^2-4\,{q_0}} \right) },\; z_2={\frac {1}{2}
 \left(- q_1-\sqrt {  q_1^2-4\, q_0}
 \right) },
$$
we have
\begin{equation}\label{fact}
Sres_0(x,y)= T(x,y,z_1)\,T(x,y,z_2).
\end{equation}

\medskip
The resultant is crucial for the study of the intersection curve that concerns us because,  due to how the torus and the quadric are defined, the following property is fulfilled: If $(x_0,y_0,z_0) \in \mathbb{C}^3$ is a common zero of  $ \mathcal{T}$ and $ \mathcal{Q}$, then $ Sres_0(x_0,y_0)=0$; and viceversa, if $  Sres_0(x_0,y_0)=0$, then for some $z_0\in  \mathbb{C}$ , $(x_0,y_0,z_0) \in  \mathbb{C}^3$ is a common zero of both $\mathcal{T}$ and $ \mathcal{Q}$. For proof, see Lemma 7.3.2. in \cite{Mishra}. By this property, one can think that the curve defined in $\mathbb{R}^2$ by $  Sres_0(x,y )=0$ is the projection of the intersection curve $ \mathcal{T}\cap \mathcal{Q}$. We will see in the next section that this is not entirely true.

Throughout the paper, we will use the notation $\widetilde{{\bf S}_0}(x,y)$ for the resultant polynomial, defined by Equation \eqref{S_0}, and the notation $\widehat{{\bf S}_0}(x,y)$ for its squarefree part.

\section{Projecting the intersection curve: the cutcurve}\label{sec2}

One of the classical tools used for studying intersections of surfaces are projections. 
Let $$\mathcal{C}_\mathbb{R}:= \{(x,y,z)\in \mathbb{R}^3  :T(x,y,z)=Q(x,y,z)=0    \}$$ be the  intersection curve of $\mathcal{T}$ and $\mathcal{Q}$ in $\mathbb{R}^3$. Given the projection $\Pi_z$, defined as follows
\[\begin{array}{*{20}{c}}
\Pi_z :&{{\mathbb{R}^3}}& \to &{{\mathbb{R}^2}}\\
{}&{\left( {x,y,z} \right)}& \mapsto &\,{\left( {x,y}  \right) ,}
\end{array}\]
the main goal of this work is to obtain as much information as possible about $\mathcal{C}_\mathbb{R}$ from its projection $\Pi_z(\mathcal{C}_\mathbb{R})$.


%
%
%
Let
$$
\Delta _{{\cal T}}(x,y) = p_0(x,y)=(R^2 + 2Rr + r^2 - x^2 - y^2)(R^2 - 2Rr + r^2 - x^2 - y^2),
$$
and let
$$
\Delta _{{\cal Q}}(x,y) = q_1(x,y)^2-4q_0(x,y)
$$
be the discriminant of $Q(x,y,z)$ with respect to $z$. In \cite{Berberich:2005:ECE:1064092.1064110},  the curves in $\R^2$ defined by $ \Delta _{{\cal T}}$ and $\Delta _{{\cal Q}}$ are called the silhouette of ${\cal T}$ and ${\cal Q}$ respectively.
Next theorem characterizes the set $\Pi_z \left( \mathcal{C}_\mathbb{R}\right)$, the projection of the real intersection curve of ${\cal T}$ and ${\cal Q}$, as a semialgebraic set in terms of the silhouettes $ \Delta _{{\cal T}}$ and $\Delta _{{\cal Q}}$.

\begin{thm}\label{projectioncurve} 
 \[
 \Pi_z \left(\mathcal{C}_\mathbb{R}\right )
 =
 \left\{ (x,y) \in \mathbb{R}^2:
\widehat{{\bf S}_0}(x,y) = 0,\Delta _{{\cal T}}(x,y)\leq 0,\Delta _{{\cal Q}}(x,y)\geq 0
 \right\} .
 \]

\end{thm}

\begin{proof}
If $(a,b)\in\Pi_z \left(\mathcal{C}_\mathbb{R} \right)$, then there exists $c\in\mathbb{R}$ such that $(a,b,c)\in {\cal T}  \cap {\cal Q} $. Thus $T(a,b,c)= Q(a,b,c)=0$, and in consequence, $T(a,b,z)$ and $Q(a,b,z)$ have a common factor, and we can conclude that $\widehat{{\bf S}_0}(a,b)=0$. Moreover, since $c$ is a real root of $Q(a,b,z)$, we should have
$\Delta _{{\cal Q} }(a,b)\geq 0$. On the other side,  $T(a,b,c)= T(a,b,-c)=0$ and the other two roots of $T(a,b,z)$ are not real, so that the discriminant is zero or negative. Then, since
$$
{\rm disc}(T(a,b,z))=4096\,\Delta_{\cal T}(a,b)R^4(a^2 + b^2)^2,
$$
and $R^4(a^2 + b^2)^2> 0$, we should have $\Delta_{\cal T}(a,b)\leq 0$.

\medskip
Conversely, if $(a,b)\in\mathbb{R}^2$ verifies $\widehat{{\bf S}_0}(a,b)=0$, then there exists $c \in \mathbb{C}$ with $T(a,b,c)=Q(a,b,c)=0$. Moreover, since $\Delta_{ \cal Q }(a,b)\geq 0$, the common root $c$ must be real. Therefore, $(a,b,c)$ is in $\mathcal{C}_\mathbb{R}$ and $(a,b)\in \Pi_z \left(\mathcal{C}_\mathbb{R}\right )$.
\end{proof}

%
%
%
%

Thus, the projection of the (real) intersection curve is bounded by the silhouettes of ${\cal T}$ and ${\cal Q}$, and therefore, it is not described just as the (real) curve defined by the squarefree part of the resultant. 

\begin{defin}
Let ${\cal T}$ be a torus and ${\cal Q}$ be a quadric in $\R^3$. We say that the semialgebraic set $\Pi_z \left( {\cal C}_{\mathbb R} \right)$ is the {\bf \emph{cutcurve}} of ${\cal T}$ and ${\cal Q}$. We also define
$$
{\cal A}_{{\cal T},{\cal Q}}:=\left\{ (x,y) \in \mathbb{R}^2:\Delta _{{\cal T}}(x,y)\leq 0, \Delta _{{\cal Q}}(x,y)\geq 0
 \right\} .
$$
 \end{defin}

\noindent

Section \ref{secej} shows several examples which illustrate these concepts.

\subsection{The cutcurve and the lifting process}\label{lifting}

The objective of this work is to obtain the maximum information of the real intersection curve from the cutcurve. So, the first question that may arise is, given a point on the cutcurve, from which point it is projection.  In this section we will see the crucial role that $q_1$ has when giving an answer to this question.

As we mentioned before, given $(x_0,y_0)\in\mathbb{R}^2$ with $\widehat{{\bf S}_0}(x_0,y_0)=0$, then there exists $z_0 \in \mathbb{C}$ with $T(x_0,y_0,z_0)=Q(x_0,y_0,z_0)=0$. This does not imply that $(x_0,y_0)\in \Pi_z \left(\mathcal{C}_\mathbb{R}\right ) $ because  $z_0$ may be complex and non-real, but this situation can not occur when $sres_1(x_0,y_0)\neq 0$.

\begin{propst}\label{z0L}
If $\left(x_0,y_0\right)$ is a real point with $\widehat{{\bf S}_0}(x_0,y_0)=0$ such that $sres_1(x_0,y_0)\neq 0$, then the point $(x_0,y_0)$ is on the cutcurve, and it is the projection of an unique point in $\mathcal{C}_\mathbb{R}$ whose z-coordinate  is given by
\begin{equation}\label{z0}
z_0 = - \frac{sres_{1,0}(x_0,y_0)}{sres_1(x_0,y_0)}\,\in \mathbb{R}\, .
\end{equation}
\end{propst}
\begin{proof}
This statement follows directly from (\ref{gcd}).
\end{proof}



Regarding the real zero set of the polynomial $sres_1$, its definition is very simple, and besides, the characteristics of the torus and quadrics will provide us with an easy characterization of the points on the cutcurve with $sres_1(x_0,y_0)=0$.

\begin{lem}\label{sres1}
$$
Sres_1 = q_1(q_1^2+p_2 - 2q_0)z +q_0q_1^2 + p_2q_0 - q_0^2 - p_0,$$
such that
$$
sres_1 =q_1(q_1^2+p_2 - 2q_0) \text{ and } sres_{1,0} =q_0q_1^2 + p_2q_0 - q_0^2 - p_0.
$$
\end{lem}
\begin{proof}
By definition, we have
$$sres_1=\left|\begin{array}{cccc}
1 & 0 & {p_2}  & 0
\\
 1  & {q_1}  & {q_0}  & 0
\\
 0 & 1  & {q_1}  & {q_0}
\\
 0 & 0 & 1  & {q_1}
\end{array}\right| \;
\text{ and } \; sres_{1,0}=\left|\begin{array}{cccc}
1 & 0 & {p_2}  & {p_0}
\\
 1  & {q_1}  & {q_0}  & 0
\\
 0 & 1  & {q_1}  & 0
\\
 0 & 0 & 1  & {q_0}
\end{array}\right|.
$$
\end{proof}

\noindent Observe that by Equation (\ref{s0eq}) and Lemma \ref{sres1}, we have
\begin{equation}
\label{expS0}
 \widetilde{{\bf S}_0}=sres_{1,0}^2+sres_1\, q_1(p_0 - q_0^2) .
\end{equation}

\bigskip
Due to these equalities, the following result enables us to characterize the points on the cutcurve which satisfy $sres_1(x_0,y_0)=0$.

\begin{propst}\label{q1sres1}
Let $ (x_0,y_0) \in  \R^2$ be a real point on the cutcurve. Then $sres_1(x_0,y_0)=0$ if and only if $q_1(x_0,y_0)=0$.
\end{propst}
\begin{proof}
If $q_1(x_0,y_0)=0$, then $sres_1(x_0,y_0)=0$ since $q_1$ is a divisor of $sres_1$ by Lemma  \ref{sres1}. To prove the converse, we need to show that if $sres_1(x_0,y_0)=0$, then $q_1(x_0,y_0)=0$. We will prove that $p_2(x_0,y_0) +q_1^2(x_0,y_0)- 2q_0(x_0,y_0)\neq 0$.

First of all, observe that the polynomial $p_2=2(R^2-r^2+x^2+y^2)$ is strictly positive over $\mathbb{R}$, because $R>r>0$. Now, since  $ (x_0,y_0) $ is on the cutcurve, we have $q_1^2(x_0,y_0)- 4q_0(x_0,y_0)\geq 0$, and, we can have either $q_0(x_0,y_0)\leq 0$, and so, $p_2(x_0,y_0) +q_1^2(x_0,y_0)- 2q_0(x_0,y_0)>0$, or $q_0(x_0,y_0)>0$, so that $p_2(x_0,y_0) +q_1^2(x_0,y_0)- 2q_0(x_0,y_0) > p_2(x_0,y_0) +q_1^2(x_0,y_0)- 4q_0(x_0,y_0)>0$. Thus we can conclude that $p_2(x_0,y_0) +q_1^2(x_0,y_0)- 2q_0(x_0,y_0)\neq 0$ and so $q_1(x_0,y_0)=0$.
\end{proof}

Thus, when $q_1$ is not identically zero, it is easy to get the points $(x_0,y_0)$ on the cutcurve with $sres_1(x_0,y_0)=0$, because they are the intersection points of the cutcurve with the line defined by $q_1=0$. Such points are characterized because each one is projection of two different real points of the intersection curve, or of a point belonging to the plane $z=0$. Also observe that if $q_1$ does not divide $\widehat{{\bf S}_0}(x_0,y_0)$, there will be only a finite number of points on the cutcurve with $q_1=0$.

On the other hand, if $q_1\equiv0$ then the quadric $\mathcal{Q}$ is symmetric with respect to the plane $z=0$. Since $\mathcal{T}$ is also symmetric with respect to $z=0$, the intersection curve will also keep this symmetry, meaning that for any $(x_0,y_0)$ vanishing $\widehat{{\bf S}_0}(x,y)$, $(x_0,y_0,\sqrt{-q_0(x_0,y_0)})$ and $(x_0,y_0,-\sqrt{-q_0(x_0,y_0)})$ are both in $\mathcal{T}\cap\mathcal{Q}$. Obviously, $(x_0,y_0)$ is on the cutcurve iff $q_0(x_0,y_0)$ is negative or equal to 0.

Moreover, remember that our quadric is real, and so, $q_1\equiv0$ is equivalent to $sres_1\equiv 0$.

\begin{lem}\label{sres1q10} Since the quadric surface $\cal Q$ is real, we have that  $sres_1\equiv 0$ if and only if $q_1\equiv 0$.  
\end{lem}
\begin{proof}
We will prove that the polynomial $p_2 +q_1^2- 2q_0$ cannot be identically zero. First of all, observe that the polynomial $p_2=2(R^2-r^2+x^2+y^2)$ is strictly positive over $\mathbb{R}$ because $R>r>0$. Now, let $(x_0,y_0,z_0)$ be a real point on $\cal Q$.
Then $q_1^2(x_0,y_0)- 4q_0(x_0,y_0)\geq 0$ and, we can have either $q_0(x_0,y_0)\leq 0$, and so, $p_2(x_0,y_0) +q_1^2(x_0,y_0)- 2q_0(x_0,y_0)>0$, or $q_0(x_0,y_0)>0$, so that $p_2(x_0,y_0) +q_1^2(x_0,y_0)- 2q_0(x_0,y_0) > p_2(x_0,y_0) +q_1^2(x_0,y_0)- 4q_0(x_0,y_0)>0$. Thus we can conclude that $p_2 +q_1^2- 2q_0$ is positive for every real point $(x_0,y_0,z_0)\in \cal Q$, and so, it cannot be identically zero.
\end{proof}

To summarize, given $(x_0,y_0)$ on the cutcurve, 
\begin{itemize}
\item When $q_1\not\equiv 0$,
\begin{itemize}
\item If $q_1(x_0,y_0)\ne 0$, then $(x_0,y_0)$ is projection of the point $(x_0,y_0,z_0)$, $z_0$ given by Equation \eqref{z0}.
\item If $q_1(x_0,y_0)= 0$, then $(x_0,y_0)$ is projection of $(x_0,y_0,\pm\sqrt{-q_0(x_0,y_0)})$.
\end{itemize}
\item When $q_1\equiv 0$, then $\mathcal{T}\cap\mathcal{Q}$ is symmetric wrt the plane $z=0$ and $(x_0,y_0)$ is projection of $(x_0,y_0,\pm\sqrt{-q_0(x_0,y_0)})$
\end{itemize}


\subsection{Relationship between the resultant and the cutcurve}\label{lifting2}

First of all, observe that by the proof of Lemma \ref{sres1q10},  $(q_1^2 +p_2  - 2q_0)(x_0,y_0)>0$ for every $(x_0,y_0)$ of the cutcurve, and so, we can ignore the conic defined by $p_2+q_1^2   - 2q_0$ when considering the resultant. Hence, we can now give the following definition:

\begin{defin}\label{def:S_0 buena}
\[{{\bf S}_0} = \frac{\widehat{{\bf S}_0}}{\mathrm{gcd}(\widehat{{\bf S}_0},p_2  +q_1^2 - 2q_0 )}\, . \]
\end{defin}

With Definition \ref{def:S_0 buena}, we have that ${\bf S}_0$ is the generator of the ideal of all polynomials vanishing at every point of the cutcurve. Thus, from this point onward, when we talk about the curve defined by the resultant, we mean the curve defined by ${\bf S}_0$, and therefore, we also redefine the cutcurve $\Pi_z \left(\mathcal{C}_\mathbb{R}\right )$ as follows:

\[
 \Pi_z \left(\mathcal{C}_\mathbb{R}\right )
 =
 \left\{ (x,y) \in \mathbb{R}^2:
 {{\bf S}_0}(x,y) = 0,\Delta _{{\cal T}}(x,y)\leq 0,\Delta _{{\cal Q}}(x,y)\geq 0
 \right\} .
 \]

Nevertheless, the cutcurve is clearly a portion of the curve defined by ${\bf S}_0$, making it relevant to determine the conditions under which they coincide.”

%
\begin{propst} \label{resultant_es_cutc}
If the polynomial $q_1$ is not identically zero,  $\gcd(q_1, {\bf S}_0)=1$ and  $\dim(\Pi_z \left( {\cal C}_ {\mathbb R} \right))=1$, then 
$$
\qquad \Pi_z \left({\cal C}_ {\mathbb R} \right)-  \text{ isolated points } =\left\{ (x,y) \in \mathbb{R}^2:
{{\bf S}_0}(x,y) = 0\right\} - \text{  isolated points. }
$$
\end{propst}
\begin{proof}
Obviously, if $(x_0,y_0)$ is on the cutcurve $\Pi_z \left( {\cal C}_ {\mathbb R} \right)$, then ${\bf S}_0(x_0,y_0) = 0$. Moreover, if it is not isolated for the cutcurve, it is also not isolated for the curve defined by ${{\bf S}_0}$. Therefore, it just remains to prove that the rhs is contained in the lhs of the equality of sets.

Suppose that $(x_0,y_0)$ is a real point that satisfies $ {{\bf S}_0}(x_0,y_0) = 0$ and it is not an isolated point on the curve defined by $ {{\bf S}_0}$. Since $\gcd(q_1,{{\bf S}_0})=1$, there are finitely many points verifying $sres_1(x,y)=  {{\bf S}_0}(x,y)=0$, so there exists a one-dimensional neighbourhood $U$ of $(x_0,y_0)$ in the curve defined by $ {{\bf S}_0}$ where the only point annihilating $sres_1$ is, if that is the case, $(x_0,y_0)$ itself. Then, by Proposition \ref{z0L}, we have $$ \left(x,y,-\frac{sres_{1,0}(x,y)}{sres_1(x,y)}\right)\in {\cal T}\cap {\cal Q} $$ for all $(x,y)\in U-\{(x_0,y_0)\}$. Therefore, $U-\{(x_0,y_0)\}$ is contained in the cutcurve, and so is its closure, including $(x_0,y_0)$, which is a limit point of $U\subset\Pi_z \left( {\cal C}_ {\mathbb R} \right)$.
\end{proof}

Examples \ref{t9-elipse2}, \ref{toro-8} and  \ref{varios-5} illustrate the above proposition, where the curve defined by the resultant lies entirely within the region ${\cal A}_{{\cal T},{\cal Q}}$. 
   
\medskip
However, the fact that $q_1 \equiv 0$ is not sufficient to conclude whether the resultant is equal to the cutcurve or not. In Example \ref{curvareal}, with $q_1\equiv 0$, the resultant curve does not coincide with the cutcurve. In Example \ref{esfera}, also with $q_1 \equiv 0$, the cutcurve coincides with the resultant.

 \begin{cor}\label{cor:parte unidimensional}
If $q_1$ is not identically zero and $\dim(\Pi_z \left( {\cal C}_ {\mathbb R} \right))=1$, then we have one of the following possibilities:

\begin{enumerate}
\item $q_1\in \R^*$ or $q_1$ does not divide ${\bf S}_0$, and $\Pi_z \left({\cal C}_ {\mathbb R} \right)-$ isolated points

\hspace*{4 cm}$\,=\left\{ (x,y) \in \mathbb{R}^2:  {\bf S}_0(x,y) = 0\right\}-$  isolated points
\item  $\deg(q_1)=1$ and $q_1$ divides ${\bf S}_0$, and $\Pi_z \left({\cal C}_ {\mathbb R} \right)-$ isolated points

\hspace*{0.8 cm} $\,=\left\{ (x,y) \in \mathbb{R}^2:  {\bf S}_0 = 0\ne q_1\right\} \cup \left\{ (x,y) \in {\cal A}_{{\cal T},{\cal Q}}  :  q_1(x,y) = 0 \right\} -$ isolated points.
\end{enumerate}
\end{cor}

In the next section, we will discuss about the isolated points and  we will specify the meaning of $q_1$ dividing ${\bf S}_0$ in Proposition \ref{multiple_components}.

\section{Characterizing singular points} \label{lifting1}
In this section, we present different results that allow us to detect all the singularities of $\mathcal{C}_\mathbb{R}$ from the information provided by the cutcurve. It would be ideal if there was a bijective mapping between singularities of the intersection curve and those of the cutcurve. But as one can easily see, the projection $\Pi_z$ does not have to be injective, so there may be singularities on the cutcurve that are  projections of regular points; and viceversa, projections of singularities on the intersection curve in the plane $z=0$ may be regular points on the cutcurve. This makes the identification of singularities in the intersection curve more challenging than one might initially expect.

 \medskip

We first introduce some definitions which play an important role in this section.
\begin{defin}
A singular point of a real plane curve is defined as a point where the partial derivatives of the square-free polynomial that defines such a curve vanish.
\end{defin}

\begin{defin} Given two plane algebraic curves $\mathcal{F}$ and $\mathcal{G}$, we say that they intersect transversally at a point $P$ if $P$ is a regular point of $\mathcal{F}$ and of $\mathcal{G}$, and if the tangent lines to $\mathcal{F}$ and $\mathcal{G}$ at $P$ are different.
\end{defin}

If the curves are defined by $f$ and $g$, then it is well known that the curves intersect transversally at $P$ if and only if the Jacobian determinant $\frac{\delta(f,g)}{\delta(x,y)}$ does not vanish at the point $P$ (see \cite{Kunz}).

\begin{defin}
Given a real space curve $\cal C$, let $I(\cal C)$ be the ideal of the curve given by $I({\cal C})= \{f\in \mathbb R[x,y,z] / f({\cal C})=0\}$. Then a point $(x_0,y_0,z_0)\in \cal C$ is singular if the rank of the Jacobian matrix of a generator system of $I(\cal C)$ is 1 or 0 after specialization at the point. If the rank is equal to 2, the point is said regular or smooth.
 \end{defin}

In our case, ${\cal C}=\mathcal{C}_\mathbb{R}$ and although we do not have the ideal of the curve, it is well known that a point on $\cal C$ where the rank of the Jacobian matrix of $T$ and $Q$ is equal to 2 is nonsingular (see for example \cite{CLO}).



To facilitate understanding of this section, we have summarized the results in the following diagram, which illustrates the different cases that arise when determining whether $(x_0, y_0)\in \Pi_z \left(\mathcal{C}_\mathbb{R}\right )$ is the projection of a singularity of $\mathcal{C}_\mathbb{R}$. 
\begin{center}
\begin{tikzpicture}[
    level 1/.style={sibling distance=8cm},
    level 2/.style={sibling distance=6cm},
    level 3/.style={sibling distance=6cm},
    every node/.style={rectangle, rounded corners, draw, align=center, text width=4cm, fill=blue!10},
    edge from parent/.style={draw, ->, thick},
    edge from parent path={(\tikzparentnode.south) -- ++(0,-4pt) -| (\tikzchildnode.north)}
]

\node[text width=2cm]{$q_1\equiv 0$ 
}
    child {node {$\Delta_{\cal T}(x_0,y_0)\ne 0$} 
        child {node[text width=6cm] {$(x_0,y_0)$ is a singular point  iff $(x_0,y_0,\pm\sqrt{-q_0(x_0,y_0)})$ are both singular points. \\Theorem \ref{singarribaabajo} and Proposition \ref{Singq_1Siemprecero_noCorona}}
        } 
      }
    child {node {$\Delta_{\cal T}(x_0,y_0)=0$}
        child {node[text width=7cm] {$\Delta_Q$ and ${\bf S_0}$ intersect non-transversally at $(x_0,y_0)$ and $\Delta_{\cal Q}$ is not a component of ${\bf S_0}$ iff $(x_0,y_0,0)$ is a singular point.\\Theorem \ref{Singq_1Siemprecero_enCorona}}
        }
     };

\end{tikzpicture}
\end{center}

\bigskip
\begin{center}
\begin{tikzpicture}[
    level 1/.style={sibling distance=8cm},
    level 2/.style={sibling distance=4cm},
    level 3/.style={sibling distance=4cm},
    every node/.style={rectangle, rounded corners, draw, align=center, text width=4cm, fill=red!10},
    edge from parent/.style={draw, ->, thick},
    edge from parent path={(\tikzparentnode.south) -- ++(0,-4pt) -| (\tikzchildnode.north)}
]

\node[text width=3cm] { $q_1(x_0,y_0)\ne 0$
}
 child {node {$\Delta_{\cal T}(x_0,y_0)\ne 0$} 
        child {node[text width=5cm] { $(x_0,y_0)$ is singular iff $(x_0,y_0,z_0)$ is singular\\Theorem \ref{singarribaabajo} and Proposition \ref{Singq_1nocero_noCorona}}
        } 
      }
    child {node {$\Delta_{\cal T}(x_0,y_0)=0$}
        child {node[text width=7cm] {$(x_0,y_0)$ and $(x_0,y_0,0)$ are both regular. \\Proposition \ref{Singq_1nocero_noCorona} }
        }
     };
    
    \end{tikzpicture}
\end{center}

 \bigskip

\begin{center}
\begin{tikzpicture}[
    level 1/.style={sibling distance=8.4cm},
    level 2/.style={sibling distance=5.5cm},
    level 3/.style={sibling distance=5.8cm},
    every node/.style={rectangle, rounded corners, draw, align=center, text width=4cm, fill=green!10},
    edge from parent/.style={draw, ->, thick},
    edge from parent path={(\tikzparentnode.south) -- ++(0,-2pt) -| (\tikzchildnode.north)}
]

\node[text width=5.5cm] {$q_1\not\equiv 0$ with $q_1(x_0,y_0)= 0$
}
    child {node {$q_1|{\bf S_0}$} 
        child {node[text width=5.5cm] { ${\bf S_0}/q_1$ vanishes at $(x_0,y_0)$ (so,  
        $(x_0,y_0)$ singular) iff one of $(x_0,y_0,\pm\sqrt{-q_0(x_0,y_0)})$ is singular.\\Proposition \ref{multiple_components}, Remarks \ref{2bi} and \ref{coverticalidad}}
        } 
      }
    child {node {$q_1\nmid \;{\bf S_0}$}
        child{node [text width=3cm]{$\Delta_{\cal T}(x_0,y_0)\ne 0$ and $(x_0,y_0)$ is singular}   
        		child {node[text width=5cm] {Both $(x_0,y_0,\pm\sqrt{-q_0})$ regular,\\ or one singular and the other regular.\\Theorem \ref{singarribaabajo} and Remark \ref{coverticalidad}}
		}
        }
        child {node[text width=3cm] {$\Delta_{\cal T}(x_0,y_0)=0$}
            	child {node [text width=5.3cm] {$(x_0,y_0,0)$ singular iff $(x_0,y_0)$ is a non-transversal intersection point of $\Delta_{\cal T}(x_0,y_0)$ and $\Delta_{{\cal Q}}(x_0,y_0)$.
	\\ Theorem \ref{cortangent2} and Corolary \ref{porfin}
	}            
        }
    }
    };
    
    \end{tikzpicture}
\end{center}
 


\newpage

In the previous diagram, we do not differentiate the case where the point on the cutcurve is isolated. In fact, if $(x_0,y_0)$ is an isolated point of the cutcurve, not necessarily isolated of the resultant curve, then it is the projection of an isolated (so singular) point of the intersection curve. Viceversa, if we have an isolated point on the intersection curve $(x_0,y_0,z_0)$ with $z_0\neq 0$, then $(x_0,y_0)$ is singular. However, if $z_0=0$, the cutcurve may be regular at $(x_0,y_0)$, because the cutcurve is semialgebraic while the intersection curve is algebraic (see Example \ref{toro9-2sheet}). In any case, this makes no difference when applying our results. 

On the other hand, if the quadric $\mathcal{Q}$ is a cone and its singular point, denoted by $P$,  lies on $\mathcal{C}_\mathbb{R}$, then $P$ is a singularity of the curve $\mathcal{T}\cap\mathcal{Q}$ by the next proposition. Moreover, if $P$ is not in the plane $z=0$, then its projection is a singular point of the cutcurve as well by Theorem \ref{singarribaabajo}; otherwise, Theorems \ref{Singq_1Siemprecero_enCorona} and  \ref{cortangent2} are applied (see Example \ref{cono}).

\begin{propst}\label{singularidadcono}
If $\mathcal{Q}$ is a real cone whose vertex $P$ lays on the torus, then the intersection curve $\Cr$ is singular at $P$.
\end{propst}


\begin{proof}

If $P $ were an isolated point on $\Cr$, then it would be singular. For the purpose of the following discussion, let’s assume it is not. 

First, note that since the cone is 3-dimensional, a general generatrix $G$ passing through $P$ is transversal to the tangent plane to the torus at $P$, and then to the torus itself. Therefore, among the four intersection points (counted with multiplicity) of $G\cap\mathcal{T}$, just one corresponds to $P$.

Under these conditions, suppose  that $\cal{T}$ and $\cal{Q}$  meet tangentially along all $\cal{C}_\mathbb{R}$.   Then the remaining 3 points in $G\cap\mathcal{T}$ would coincide and, given the odd multiplicities, be real.
%
%
%
%
%
Now 
consider a linear parametrization $\varphi(s,t)$ of the cone, where the generatrices are images of vertical lines, the vertex $P$ corresponds to the horizontal axis and the other horizontal lines parametrize ellipses through the stereographical projection. Then, the composition of the equation of the torus and $\varphi$ gives an equation on $s,t$ of type $tp(s,t)=\alpha(s)t^4+\beta(s)t^3+\gamma(s)t^2+\delta(s)t$. Since, for a general $s_0$, the restriction must factor as $\left(\sqrt[3]{\alpha(s_0)}t+\sqrt[3]{\delta(s_0)}\right)^3t$, both the resultant and first subresultant of $p$ and $p_t$ wrt the variable $t$ are generically, so identically zero. Thus, $p(s,t)$ is a cube.
This means that the intersection is given by a triple component corresponding to $p(s,t)$ and $P$ corresponding to $t=0$. Since $P$ is not isolated, the triple component must contain $P$. 
 
Due to the degree constrains, such component can only be an irreducible conic, as it is well known that the torus does not contain real lines. However, since a cone does not contain irreducible conics passing through the vertex, we have a contradiction, proving that the intersection of a torus and a cone through its vertex cannot be generically tangential.

Therefore, the intersection must have a simple real component. Under these conditions, the multiple components have degree at most 3. Given that the torus is bounded, this means that such components must be conics. Since, again, a cone does not contain irreducible conics passing through the vertex, we conclude that such vertex lays exclusively on simple components of the intersection. Since the intersection is generically transversal on simple components, the zero gradient of the cone at the vertex provides that $P$ is singular of $\mathcal{C}_\mathbb{R}$.
\end{proof}

In the remaining part of the section, 
we will need the following lemma, derived from \cite[Lemma 5.1]{Mourrain2009} applied to the affine setting to $T$ and $Q$, which characterizes the derivatives of the resultant $\widetilde{{{\bf S}_0}}$.

\begin{lem}\label{singular20}
We have
$$\frac{\partial {\widetilde{{\bf S}_0}}}{{\partial x}}\left( {{x },{y }} \right) =
sres_1(x ,y )\cdot
{\left| {\begin{array}{cc}
T_x(x ,y ,z )&T_z(x ,y ,z )\\
Q_x(x ,y ,z)&Q_z(x ,y ,z )
\end{array}} \right|} \, \mod(T,Q),$$
and
$$\frac{\partial {\widetilde{{\bf S}_0}}}{{\partial y}}\left( x,y \right) =
sres_1(x,y)\cdot
{\left| {\begin{array}{cc}
T_y(x,y,z)&T_z(x,y,z)\\
Q_y(x,y,z)&Q_z(x,y,z)
\end{array}} \right|}  \mod(T,Q)$$
\end{lem}

\medskip
The following two results are an immediate consequences of the above lemma. Observe that, in the following corollary, the hypotheses imply that there is no vertical tangency.

\begin{cor}\label{cortangent1} Let $(x_0,y_0,z_0)\in \mathcal{C}_\mathbb{R}$ with $q_1(x_0,y_0)\neq 0$. Then, the Jacobian matrix of $T$ and $Q$ at $(x_0,y_0,z_0)$ has rank one if and only if the partial derivatives of $\widetilde{{\bf S}_0}$ vanish at $(x_0,y_0)$.
\end{cor}
  \begin{proof}
One direction of the proof is obvious.

To prove the converse, let's see that if the partial derivatives of $\widetilde{{\bf S}_0}$ vanish at $(x_0,y_0)$, then the Jacobian matrix has rank 1. Since ${T_z}\left( {{x_0},{y_0},{z_0}} \right)=2z_0(2z_0^2 + p_2)$ with $p_2>0$ and ${{Q_z}\left( {{x_0},{y_0},{z_0}} \right)}=q_1 + 2z_0$, by the hypotheses these two derivatives do not vanish simultaneously, that is, there is no vertical tangency. Then, by applying Lemma \ref{singular20}, where the partial derivatives of $\widetilde{{\bf S}_0}$ are multiples of two maximal minors of $J$, we can deduce that the Jacobian matrix evaluated at $( {{x_0},{y_0},{z_0}})$ must have rank equal to one.
%
%

%
%
%
 \end{proof}

Note that the fact of having the Jacobian matrix with rank 1 and derivatives of $\widetilde{{\bf S}_0}$ zero, does not imply the existence of singularities. See for instance Example \ref{ej_toro_cono_vert}, where  $q_1$ is a nonzero constant and  there is tangency along a component of $\mathcal{C}_\mathbb{R}$. The rank of the Jacobian matrix is 1 throughout the entire component, and the derivative of $\widetilde{{\bf S}_0}$ is zero on its projection.
 
The following result allows to compute singularities of the cutcurve in a simple way if the resultant is squarefree.
 	
\begin{cor}\label{q10S}
Let $(x_0,y_0)$ be a point on the cutcurve. If $q_1(x_0,y_0)=0$, then
$$
\frac{{\partial {\widetilde{{\bf S}_0}}}}{{\partial x}}\left( {{x_0},{y_0}} \right) =0 ,\qquad  \frac{{\partial {\widetilde{{\bf S}_0}}}}{{\partial y}}\left( x_0,y_0 \right) = 0 .
$$
In particular, if the resultant polynomial is squarefree,  or $(x_0,y_0)$ is in one simple component of the resultant polynomial, then the point $(x_0,y_0)$ is a singular point on the cutcurve.
\end{cor}

\medskip

However, the previous corollaries do not relate the point $(x_0,y_0)$ to the nature of its lifting point. This issue is more complex than expected. When studying the correspondence between the singularities of the intersection curve and the cutcurve, the first question that arises is whether all the singularities of the intersection curve project onto singularities in the cutcurve. This is the case if the $z$-coordinate is not zero.


\begin{thm}\label{singarribaabajo}
Let $(x_0,y_0,z_0)$ be a singular point of $\Cr$ with $z_0\ne 0$. Then, $(x_0,y_0)$ is a singularity of the curve defined by ${\bf S}_0$.
\end{thm}

\begin{proof}
The ideal of the intersection curve contains both ${{\bf S}_0}$ and $T$. Since our point is a real singularity of $\mathcal{T}\cap\mathcal{Q}$, the Jacobian matrix of $T$ and ${{\bf S}_0}$ at $(x_0,y_0,z_0)$ has rank one. As ${{\bf S}_0}$ is independent of $z$, we have $\frac{\partial{{\bf S}_0}}{\partial z}\equiv0$. Furthermore, since $z_0\ne0$ and the point is real, the tangent plane to $\mathcal{T}$ at $(x_0,y_0,z_0)$ is not vertical, implying $\frac{\partial T}{\partial z}(x_0,y_0,z_0)\ne0$. Therefore, for the minors of the Jacobian matrix of $T$ and ${{\bf S}_0}$ involving last column to vanish, $\frac{\partial{{\bf S}_0}}{\partial x}$ and $\frac{\partial{{\bf S}_0}}{\partial y}$ must be zero at $(x_0,y_0,z_0)$ (i.e. at $(x_0,y_0)$).
\end{proof}

Examples \ref{toro-8} and \ref{YVC} illustrate the above theorem. Conversely, given a singular point on the cutcurve, we must differentiate whether it only has one lift or two, and whether or not it is on $\Delta_{\cal T}$ in order to establish specific results. 

\medskip
We shall begin by examining what happens when $q_1\equiv0$, that is, when the quadric $\mathcal{Q}$ is symmetric wrt the plane $z=0$.

\subsection{Case $q_1\equiv 0$ }


\begin{propst}\label{Singq_1Siemprecero_noCorona}
Suppose $q_1\equiv0$. Let $(x_0,y_0)$ be a singular point of the cutcurve with $\Delta_\mathcal{T}(x_0,y_0)\neq 0$. Then $(x_0,y_0)$ is projection of two singular points on $\Cr$.
\end{propst}

\begin{proof}
Note that, since $q_1\equiv0$, by Equations (\ref{s0eq}) and (\ref{expS0}), ${\bf S_0}$ is either the polynomial $sres_{10}$,
$$
 sres_{10} =p_2q_0  - q_0^2 - p_0   \in \R[x,y],
 $$ 
or a divisor of $sres_{10}$. Moreover, $sres_{10}$ is the remainder of dividing $-T$ by $Q$ as polynomials in $\R(x,y)[z]$, which means that $T$ is an algebraic combination of $Q$ and ${\bf S_0}$. Then, the real intersection of the vertical cylinder defined by ${\bf S_0}$ and $\mathcal{Q}$ is $\Cr$ .

On the other hand, $(x_0,y_0)$ is projection of $(x_0,y_0,\pm\sqrt{-q_0(x_0,y_0)})$ with $q_0(x_0,y_0)\neq 0$. Fix $z_0=\sqrt{-q_0(x_0,y_0)}$. Since ${\bf S_0}$ is squarefree and $z_0\neq 0$, a general point of the intersection of the vertical cylinder and the quadric in a neighbourhood of $(x_0,y_0,z_0)$ in $\C^3$ will not be singular and will have the $z$-coordinate different from zero. Moreover, $\dfrac{\partial Q}{\partial z}=2z$ and thus, the Jacobian matrix of $Q$ and ${\bf S_0}$ will have rank 2. This means that the ideal $\langle Q,{\bf S_0}\rangle$ is its own radical in such neighbourhood, so that it is the ideal of $\Cr$. Now, since $(x_0,y_0)$ is a singular point on the cutcurve, the Jacobian matrix of $\langle Q,{\bf S_0}\rangle$ evaluated at $(x_0,y_0,z_0)$ has rank one and we can conclude that the point $(x_0,y_0,z_0)$ is singular. The same argument is valid for $z_0=-\sqrt{-q_0(x_0,y_0)}$. 

\end{proof}

Observe that under the hypothesis of Proposition \ref{Singq_1Siemprecero_noCorona}, the singular point $(x_0,y_0)$ is actually projection of two different singular points of  $\mathcal{C}_\mathbb{R}$. See Example \ref{cuentasjge-cil}. This situation is only possible when $q_1$ is identically 0.

\begin{propst}\label{Symmetric_q1cero}
Let $(x_0,y_0,z_0)$ and $(x_0,y_0,-z_0)$ be two different real points in $\Cr$, such that the rank of the Jacobian matrix of $T$ and $Q$ at both points is equal to one. Then $q_1\equiv 0$.
\end{propst}
\begin{proof}
Recall that  $Q(x,y,z)= z^2 + (ex + fy + i)z + ax^2 + by^2 + dxy + gx + hy + j$. We call the points $A=(x_0,y_0,z_0)$ and $B=(x_0,y_0,-z_0)$. By hypothesis, $z_0\neq 0$ and the Jacobian matrix of $T$ and $Q$ at both points, $A$ and $B$, has rank one.
Then, the Gr\"{o}bner basis of the ideal generated by the polynomials $T(A)$, $Q(A)$, $T(B)$, $Q(B)$, $Q_x(A)T_z(A)-Q_z(A)T_x(A)$, $Q_z(A)T_y(A)-Q_y(A)T_z(A)\; $, $Q_x(A)T_y(A)-Q_y(A)T_x(A),\; $ $Q_x(B)T_z(B)-Q_z(B)T_x(B), $ $Q_z(B)T_y(B) $ $-$ $Q_y(B)T_z(B)$, $Q_x(B)T_y(B)-Q_y(B)T_x(B)$ and $R r z_0 t-1$ wrt the graded reverse lexicographic order, contains $ex_0$, $fx_0$, $ey_0$, $fy_0$ and $i$. Since the points are real, that implies $e=f=i=0$ and so $q_1$ is identically zero. 
\end{proof}
Therefore, when $q_1\not\equiv 0$, a point of the cutcurve can never be lifted to two singular points.

\begin{cor}\label{singnotwosing}
Let $(x_0,y_0)$ be a point on the cutcurve such that $q_1(x_0,y_0)=0$ and $q_1\not\equiv 0$. Then
$(x_0,y_0)$ cannot be projection of two different singular points of $\mathcal{C}_\mathbb{R}$.  \end{cor}
%
%
%

\medskip
The previous results describe cases where there is a mapping between singularities on $\mathcal{C}_\mathbb{R}$ and the cutcurve. However, as we mentioned before, it is not always the case. Next, we will see that the only singularities of $\cal C_\mathbb{R}$ whose projections might not be singularities of the cutcurve are actually in the plane $z=0$.

\begin{thm}\label{Singq_1Siemprecero_enCorona}
Suppose $q_1\equiv0$ and let $(x_0,y_0,0)\in\Cr$.
Then $(x_0,y_0,0)$ is a singular point of $\Cr$ iff $ (x_0,y_0) $ is a non-transversal intersection of $\Delta_\mathcal{Q}$ and ${\bf S_0}$ and $\Delta_\mathcal{Q}$ is not a component of ${\bf S_0}$.
\end{thm}

\begin{proof}
Notice that, in this particular case, $\Delta_Q=-4q_0$ and $\Cr$ is the real intersection of $\mathcal{Q}$ and the cylinder over ${\bf S_0}$ (see the proof of Proposition \ref{Singq_1Siemprecero_noCorona}).

Let $(x_0,y_0,0)$ a singularity of $\Cr$. Then $(x_0,y_0)$ is obviously in $\Delta_{\cal Q}\cap\Delta_{\cal T}$.
Suppose now that $\Delta_{\cal Q}$ and ${\bf S_0}$ intersects transversally at $(x_0,y_0)$. Then the Jacobian matrix of $\Delta_\mathcal{Q}$ and ${\bf S_0}$ has rank 2 at $(x_0,y_0)$ and so, the Jacobian matrix of $Q=z^2+q_0$ and ${\bf S_0}$ evaluated at $(x_0,y_0,0)$ has also rank 2. Hence the point $(x_0,y_0,0)$ is a smooth point of the intersection of $\mathcal{Q}$ and the vertical cylinder over ${\bf S_0}$ (i.e. $\mathcal{T}\cap\mathcal{Q}$), which is a contradiction.


Suppose now that $\Delta_\mathcal{Q}$ is a component of ${\bf S_0}$. Hence, for every point $(x_1,y_1)$ on the curve defined by $\Delta_\mathcal{Q}$, we have $(x_1,y_1,0) \in \mathcal{T}\cap\mathcal{Q}$, and so $\Delta_\mathcal{T}(x_1,y_1)=0$. Consequently, $\Delta_\mathcal{Q}$ is also a component of $\Delta_\mathcal{T}$, so it is a circle centered at the origin.

Recall that  ${\bf S_0}$ is the squarefree part of $sres_{10}$,
$
 sres_{10} =p_2q_0  - q_0^2 - p_0
$, $p_0=\Delta_{\cal T}$ and $p_2=2(R^2-r^2+x^2+y^2)$.
Therefore, $\Delta_\mathcal{Q}$, $p_2$ and the other factor of $\Delta_\mathcal{T}$ are of type $kx^2+ky^2+$ constant, with $k\in\R$. This means that the quotient of $sres_{10}$ divided by $\Delta_\mathcal{Q}$ is also of that type, so any other component of ${\bf S_0}$, if real, is a circle centered at the origin, without real points in common with $\Delta_\mathcal{Q}$. This means that the only component of $\Cr$ passing through $(x_0,y_0,0)$ is the circle $\{\Delta_\mathcal{Q}=0 \}\cap \{z=0\}$ in $\R^3$, which is smooth. This yields another contradiction.

Finally, suppose that $\Delta_\mathcal{Q}$ and ${\bf S_0}$ intersect non-transversally at $(x_0,y_0)$ and $\Delta_\mathcal{Q}$ is not a component of ${\bf S_0}$.
Then, $(x_0,y_0)$ is an isolated point of ${\bf S_0}\cap\Delta_\mathcal{Q}$, so that $(x_0,y_0,0)$ is an isolated point of $\Cr\cap\{z=0\}$. Now, since ${\bf S_0}$ is squarefree, there is a neighbourhood $U$ of $(x_0,y_0)$ where the gradient of ${\bf S_0}$ does not vanish outside $(x_0,y_0)$ and $U\cap{\bf S_0}\cap\Delta_\mathcal{Q}=\{(x_0,y_0)\}$. Then, the Jacobian matrix of $Q$ and ${\bf S_0}$ has rank two at any point in  $\mathcal{T}\cap\mathcal{Q}\cap(U\times\R)$ but $(x_0,y_0,0)$. Then, either $(x_0,y_0,0)$ is an isolated point of $\Cr$, or it is an isolated point of the intersection of $\Cr$ with the variety defined by the $2\times2$ minors of the Jacobian matrix of $Q$ and ${\bf S_0}$. Both cases imply that the point is singular for $\Cr$.
\end{proof}

Examples \ref{curvareal} and \ref{toro9-2sheet} illustrate  Theorem \ref{Singq_1Siemprecero_enCorona}. Next, we see what happens when $q_1\not\equiv 0$. 

\subsection{Case $q_1\not\equiv 0$}

Given a point $(x_0,y_0)$ on the cutcurve, recall that if $q_1\not\equiv 0$ and $q_1(x_0,y_0) \ne0$, then it is projection of a unique point of $\cal{T}\cap \cal{Q}$. In this case, we obtain the following result.

\begin{propst}\label{Singq_1nocero_noCorona}
%
%
Let $(x_0,y_0)$ be a point of the cutcurve with $q_1(x_0,y_0) \ne0$. Then,
\begin{enumerate}
\item If $\Delta_\mathcal{T}(x_0,y_0)= 0$, then $(x_0,y_0)$ and $(x_0,y_0,0)$ are both regular.
\item If $\Delta_\mathcal{T}(x_0,y_0)\ne0$ and $(x_0,y_0)$ is a singular point, then, it is the projection of a singular point of $\Cr$.
 \end{enumerate}
\end{propst}

\begin{proof}  
\begin{enumerate}
\item 
First observe that $q_0(x_0,y_0)=p_0(x_0,y_0)=0$. The hypothesis $q_1(x_0,y_0)\neq 0$ implies that the degree of the gcd of $T(x_0,y_0,z)$ and $Q(x_0,y_0,z)$ is equal to $1$. In fact, since $\Delta_\mathcal{T}(x_0,y_0)=0$, the gcd is equal to $z$. Then, since $z=0$ is double root of $T(x_0,y_0,z)=z^4+p_2(x_0,y_0)z^2$, $z=0$ must be a simple root of $Q(x_0,y_0,z)$. This implies that $(x_0,y_0,0)$ is regular for $Q$ and that the tangent plane to $\mathcal{Q}$ at $(x_0,y_0,0)$ is not vertical ($Q_z\ne0$), so the intersection with $\mathcal{T}$ is transversal, and $(x_0,y_0,0)$ is a regular point of $\mathcal{T}\cap\mathcal{Q}$. Besides, by Corollary \ref{cortangent1}, the partial derivatives of $\widetilde{{\bf S}_0}$ do not vanish at $(x_0,y_0)$ simultaneously, which means that $(x_0,y_0)$ is a regular point of the resultant curve, so it is also regular for the cutcurve.

\item
If $(x_0,y_0)$ is an isolated point of the cutcurve, then it must be the projection of a singular point. Otherwise, there is a neighbourhood $U$ of $(x_0,y_0)$ where we can lift the cutcurve using Proposition \ref{z0L}. This means that, in $U\times\R$, the intersection curve is defined by ${\bf S_0}(x,y)=0$ and Equation \eqref{z0}. Observe that the Jacobian matrix of both equations has rank two if $q_1(x,y)\ne0$ and one of the partial derivatives of ${\bf S_0}$ does not vanish. Thus, since ${\bf S_0}$ is squarefree, this happens for a general point of $(U\times\R)  \cap \Cr$. This means that the ideal of the intersection curve in the open subset $U\times\R$ is generated by ${\bf S_0}$ and Equation \eqref{z0}, whose Jacobian matrix has rank one at $(x_0,y_0,z_0)$, $z_0$ given by Equation \eqref{z0}, because the partial derivatives of ${\bf S_0}$ vanish. This means that $(x_0,y_0,z_0)$ is singular in $\Cr$.
\end{enumerate}
\end{proof}



 Example \ref{toro-8} illustrates the above results. 
%
%
Hence, when there is only one lift, the situation is quite clear. However, it deserves much more care to study the case $q_1\not\equiv 0$ and $q_1(x_0,y_0) =0$. With this objective, we must introduce results that help us identify when there is tangency along a component of $\mathcal{C}_\mathbb{R}$. 


\begin{thm}\label{cortangent2}
Let $(x_0,y_0,0)\in\Cr$. Then, the following are equivalent:
 \begin{enumerate}
 \item The point $(x_0,y_0)$ is a non-transversal intersection point of $\Delta_{\cal T}(x_0,y_0)$ and $\Delta_{{\cal Q}}(x_0,y_0)$.
 \item The Jacobian matrix of $T$ and $Q$ at $(x_0,y_0,0)$ has rank one.
 
\end{enumerate}
 
%
%

\end{thm}

\begin{proof}
Observe that, since $(x_0,y_0,0)\in\mathcal{T}\cap\mathcal{Q}$, it follows that $\Delta_{{\cal T}}(x_0,y_0)={{\bf S}_0}(x_0,y_0)=q_0(x_0,y_0)=0$.

The point $(x_0,y_0)$ is a non-transversal intersection of $\Delta_{{\cal T}}$ and $\Delta_{{\cal Q}}$ if and only if
$\Delta_{{\cal T}}(x_0,y_0)= \Delta_{{\cal Q}}(x_0,y_0)=0$ and
\begin{multline}\label{td}
\left| {\begin{array}{*{20}{c}}
{\Delta_{{{\cal T}}_x}\left(x_0,y_0 \right)}&{\Delta_{{{\cal Q}}_x}\left( x_0,y_0 \right)}\\
{\Delta_{{{\cal T}}_y}\left( x_0,y_0\right)}&{\Delta_{{{\cal Q}}_y}\left( x_0,y_0 \right)}
\end{array}} \right|
=\\
=-8 (R^2 + r^2 - x_0^2 - y_0^2)\left( (fx_0-ey_0)q_1(x_0,y_0)+4ax_0y_0 - 4bx_0y_0 - 2dx_0^2 + 2dy_0^2 + 2gy_0 - 2hx_0\right)=0.
\end{multline}
Since $0<r<R$, we have $|R^2 + r^2 - x^2 - y^2|=2Rr\neq 0$ on the curve defined by $\Delta_{\cal T}$. Then, the above conditions are equivalent to $S_0(x_0,y_0)=p_0(x_0,y_0)=q_1(x_0,y_0)=q_0(x_0,y_0)=2ax_0y_0 - 2bx_0y_0 - dx_0^2 + dy_0^2 + gy_0 - hx_0 =0$.

On the other hand,
$$\left| {\begin{array}{*{20}{c}}
{{T_x}\left( x_0,y_0,0 \right)}&{{T_z}\left( x_0,y_0,0 \right)}\\
{{Q_x}\left( x_0,y_0,0 \right)}&{{Q_z}\left( x_0,y_0,0 \right)}
\end{array}} \right|= -4x_0(R^2 + r^2 - x_0^2 - y_0^2)q_1(x_0,y_0),$$
$$
\left| {\begin{array}{*{20}{c}}
{{T_y}\left( x_0,y_0,0 \right)}&{{T_z}\left(x_0,y_0,0 \right)}\\
{{Q_y}\left( x_0,y_0,0\right)}&{{Q_z}\left( x_0,y_0,0 \right)}
\end{array}} \right|=-4y_0(R^2 + r^2 - x_0^2 - y_0^2)q_1(x_0,y_0),
$$
and
$$
\left| {\begin{array}{*{20}{c}}
{{T_x}\left( {{x_0},{y_0},{0}} \right)}&{{T_y}\left( {{x_0},{y_0},{0}} \right)}\\
{{Q_x}\left( {{x_0},{y_0},{0}} \right)}&{{Q_y}\left( {{x_0},{y_0},{0}} \right)}
\end{array}} \right|=-4(R^2 + r^2 - x_0^2 - y_0^2 )(2ax_0y_0 - 2bx_0y_0 - dx_0^2 + dy_0^2 + gy_0 - hx_0).
$$
Therefore, it is clear that these determinants vanish and so, the rank of the Jacobian matrix is 1, if and only if $(x_0,y_0)$ is a non-transversal intersection point of $\Delta_{\cal T}(x_0,y_0)$ and $\Delta_{{\cal Q}}(x_0,y_0)$.

Observe that the fact that the rank of the Jacobian matrix is 1 means that the tangent plane at $(x_0,y_0,0)$ to the torus is contained into the tangent plane at $(x_0,y_0,0)$ to the quadric.
%
\end{proof}

To better understand what Theorem \ref{cortangent2} means, see Examples \ref{varios-5} and \ref{cuentasjge-cil}. In Example \ref{varios-5}, with $q_1\not\equiv 0$, the point on the cutcurve that satisfies Theorem \ref{cortangent2} is singular and the lifting point is also singular on the intersection curve.
In contrast, in Example \ref{cuentasjge-cil}, with $q_1\equiv 0$, the points on the cutcurve that satisfy the hypotheses of Theorem \ref{cortangent2} are not singular, and neither are their liftings. In fact, the hypotheses of Theorem
\ref{Singq_1Siemprecero_enCorona} are not satisfied. However, there is tangency along a component of the space curve. Proposition \ref{multiple_components} and Theorem \ref{lma:mult_comp} explain this situation.


%
%
%


\begin{propst}\label{multiple_components}Suppose $q_1\not\equiv0$. Let $f\in\mathbb{R}[x,y]$ be an irreducible multiple factor of $\widetilde{{\bf S}_0}$ (i.e. $f^2|\,\widetilde{{\bf S}_0}$) contributing a one dimensional subset to the cutcurve. Then:
\begin{enumerate}
\item If $f=kq_1$, with $k\in\mathbb{R}$, then $q_1(0,0)=0$, and the segment of the line $q_1$ in ${\cal A}_{{\cal T},{\cal Q}}$ corresponds to the projection of a meridian circle (vertical or cross-sectional circle) of $\Cr$.
\item Otherwise, the curve defined by $f$ can be lifted via Equation \eqref{z0} to a component of $\Cr$ along which the intersection is tangential.
\end{enumerate}

\end{propst}
\begin{proof}
First, observe that if $q_1(0,0)\ne 0$, the vertical plane defined by $q_1=0$ intersects $\mathcal{T}$ in an irreducible quartic (see \cite{Kimtesis} and \cite{Moroni}).
Since its intersection with $\mathcal{Q}$ is a conic, then the intersection of both vertical curves must be finite, which discards the possibility of $f$ being a scalar times $q_1$.

So, in the first case, the vertical plane given by $q_1=0$ intersects $\mathcal{T}$ in two circles and $\mathcal{Q}$ in a conic. Since $q_1$ is a factor of the resultant, the aforementioned conic must be one of the two circles.

In the second case, the curve defined by $f$ shares with the line defined by $q_1$ just a finite amount of points. Since, by Proposition \ref{q1sres1}, the intersection of $sres_1$ with the cutcurve lies in the line defined by $q_1$, this means that $sres_1$ is generically nonzero in the curve defined by $f$, so the lifting is as the statement describes. The tangency is immediate consequence of Corollary \ref{cortangent1}, since our quadric does not admit more than one singularity. Furthermore, the next theorem, Theorem \ref{lma:mult_comp}, asserts that the curve defined by $f$ is a parellel. 
\end{proof}

 
\begin{rem}\label{2bi}By applying Proposition \ref{multiple_components}, a segment of $q_1$ lifts to a meridian circle, which is smooth. Consequently, if ${\bf S}_0/q_1 $ vanishes at a point $(x_0,y_0)$ with $q_1(x_0,y_0)=0$, then this point is a singularity of the cutcurve, and, by Corollary \ref{singnotwosing}, one of the two points $(x_0,y_0,\pm\sqrt{-q_0(x_0,y_0)})$ is also singular, as it belongs to several components of $\mathcal{C}_\mathbb{R}$. By checking at which point the range of the jacobian matrix is maximum, we are able to identify where the singularity is. %
Obviously, if $q_0(x_0,y_0)=0$, $(x_0,y_0,0)$ is singular.
\end{rem}

\medskip

Regarding the multiple components, we can be even more precise thanks to the results in \cite{bottema}. 

\begin{thm}\label{lma:mult_comp}
Let $\mathcal{D}$ be a multiple component of $\Cr$. Then $\mathcal{D}$ is either a parallel or a meridian circle.
\end{thm}
\begin{proof}
Knowing that real algebraic curves in the torus have even degree (se, e.g. \cite[Section 2]{bottema})\footnote{creo que esto lo decimos nosotros en alg\'un momento, pero por si acaso} and that the degree of $\mathcal{T}\cap\mathcal{Q}$ is at most 8, a multiple component can only be a conic or a quartic. Note moreover that, if $\mathcal{D}$ is a multiple component, the quadric and the torus are tangent along $\mathcal{D}$.

In the case of conics, Bottema and Primrose prove in \cite{bottema} that $\mathcal{D}$ is either a parallel circle, a meridian circle or a Villareceau circle. The first cases are shown possible in Examples \ref{ej_toro_cono_vert} and \ref{cuentasjge-cil} \footnote{6.2 y 6.9 en la versi\'on que yo tengo}, so we need to discard the Villarceau case. Note that the plane containing a Villarceau circle cuts the torus in another such circle. These two circles are well known to share two points $P$ and $Q$, where such plane is tangent to $\mathcal{T}$. However, a tangent plane to a quadric intersects it in two lines (the only possible singular conic), so no quadric containing the Villarceau circle shares tangent plane with the torus at $P$ or $Q$, which means that, there, the intersection is transversal and the Villarceau circle is a simple component.

To show that a quadric and a torus cannot be tangent along an irreducible quartic curve, we take advantage of the classification in \cite[Section 4]{bottema}. According to such classification, an irreducible real quartic in the torus belongs to one of four families: I, II, II' and III.

Curves of type I are images by the parametrtization
\begin{equation}\label{eq:parametrizacion}
x=-{\frac { \left(  \left( R+r \right) {v}^{2}+R-r \right)  \left( u-1
 \right)  \left( u+1 \right) }{ \left( {u}^{2}+1 \right)  \left( {v}^{
2}+1 \right) }},
y=2\,{\frac {u \left(  \left( R+r \right) {v}^{2}+R-r
 \right) }{ \left( {u}^{2}+1 \right)  \left( {v}^{2}+1 \right) }},
 z=2\,{
\frac {rv}{{v}^{2}+1}}
\end{equation}
of hyperbolas whose asymptotes are parallel to the axes, so of equation $\alpha uv+\beta u+\gamma v+\delta=0$ for certain $\alpha,\beta,\gamma,\delta \in\R$.
Composing with the parametrization \eqref{eq:parametrizacion} and implicitizing, we find that the only (according to \cite{bottema}) quadric through the curve has:
\[q_1={\frac { 2\left( \alpha \,\gamma+\beta \,\delta  \right) x+ \left( {\alpha }^{2}+{\beta }^{2}-{\gamma}^{2}-{\delta }^{2} \right) y}{\alpha \,\delta -\beta \,\gamma}}
\]
and
\begin{multline*}
q_0=\\{-\frac {\left(\alpha \,\delta -\beta \,\gamma \right) {x}^{2}- 2\left(\alpha \,\delta +\beta \,\gamma\right)r x +\left( \alpha \,\delta -\beta \,\gamma \right) {y}^{2}-2\left( \alpha \,\beta -\gamma\,\delta \right) ry - (R^2-r^2)(\alpha \,\delta -\beta \,\gamma)}{\alpha \,\delta -\beta \,\gamma}
}
\end{multline*}
and cuts the torus in both our original quartic and the image by \eqref{eq:parametrizacion} of the hyperbola of equation $\alpha \,{w}^{2}+\beta \,vw-\gamma\,uw-\delta \,uv=0$. The only way to have the quadric to be tangent along the full curve to the torus is having both hyperbolas being the same, equivalently, having the matrix
\[\left(\begin{array}{cccc}
\alpha&\beta&\gamma&\delta\\ 
\delta&\gamma&-\beta&-\alpha
\end{array}\right)\]
with rank 1, which cannot happen over the reals. Therefore, a quartic of type I cannot be the fully tangent intersection of a quadric and the torus.

According to \cite{bottema}, types II and II' are two disjoint families and, given a member of one of these families, the only quadric through such curve contains exactly one curve of the other family. Thus, the intersection of such quadric and the torus is the pair of quartic curves, and never one of them with multiplicity 2. This proves that tangency of a quadric and a torus along a quartic curve of type II or II' is not possible.

Finally, according to \cite{bottema}, curves of type III are spherical. This category includes plane quartics (if we consider the plane as the infinite radius sphere), but since no plane quartic can be contained in a nondegenerate quadric, this case is impossible. On the other side, the intersection of a sphere of center $(\alpha,\beta,\gamma)\in\R^3$ and radius $\delta$ with the torus in the full projective space, again according with \cite{bottema}, consists in an isotropic circle at infinity with multiplicity 2 and the quartic curve. Therefore, the ideal generated by the affine equations of the sphere and the torus is the ideal of the quartic curve. The only way to have a quadric $\mathcal{Q}$ being tangent to the torus along the quartic curve is that the square of the equation of the sphere in the coordinate ring of the torus $\R[x,y,z]/T$ is represented by a degree two polynomial (the equation of $\mathcal{Q}$). However, when we reduce the square of $Q_{\mathbb{S}}=(x-\alpha)^2+(y-\beta)^2+(z-\gamma)^2-\delta^2$ by the equation of the torus, we get a cubic polynomial whose degree 3 coefficients are just $-4\alpha$, $-4\beta$, $-4\gamma$ (note that the degree 4 homogeneous component of $T$ and $Q_\mathbb{S}^2$ coincide and that $T$ has no degree 3 homogeneous component). Therefore, the only way to have such reduction being of degree 2 is that the center of the sphere is the center of the torus, which implies that the quartic curve is, in fact, a pair of parallel circles, counted with multiplicity, instead of an irreducible one. Thus, a quartic curve of type III cannot be the tangential intersection of a quadric and the torus.
\end{proof}

This result, together with Theorem \ref{cortangent2}, completes our characterization of singularities. Note that if $q_1$ does not divide ${\bf S_0}$, so there is a double tangency in one of the parallels in the plane $z=0$, or in one of the meridians,  then  $q_1$  must be identically zero.

\begin{cor}\label{porfin} Let $(x_0,y_0)$ be a point on the cutcurve. 
If $q_1\not\equiv 0$, $q_1\not|\,\,{\bf S_0}$,  $q_1(x_0,y_0)=0$, $\Delta_{\cal T}(x_0,y_0)= 0$, then $(x_0,y_0,0)$ is singular iff $(x_0,y_0)$ is a non-transversal intersection point of $\Delta_{\cal T}(x_0,y_0)$ and $\Delta_{{\cal Q}}(x_0,y_0)$.  
\end{cor}
\begin{proof} 
If $(x_0,y_0,0)$ is singular, the Jacobian matrix of $T$ and $Q$ has rank one at that point, so Theorem \ref{cortangent2} ensures that $(x_0,y_0)$ is a non-transversal intersection point of $\Delta_{\cal T}(x_0,y_0)$ and $\Delta_{{\cal Q}}(x_0,y_0)$.

For the other implication, we first note that, if $(x_0,y_0,0)$ is isolated, then it is singular. 
Otherwise, we can conclude that $(x_0,y_0,0)$ does not belong to a multiple component of $\Cr$. Note that, according to Theorem \ref{lma:mult_comp}, the only multiple components of the intersection are meridian and parallel circles. Therefore, if $(x_0,y_0,0)$ belongs to a multiple component, such component must be a parallel circle (and, then, an equator, since $z_0=0$). However, this would mean that the quadric has vertical tangent plane along the full equator circle, meaning that $q_1$ vanishes along the full circle and, therefore, is identically zero, against our hypotheses.

Because the components of $\Cr$ through $(x_0,y_0,0)$ are simple, the intersection is generically transversal near our point. Since, by Theorem \ref{cortangent2}, the Jacobian matrix of $T$ and $Q$ has rank one at $(x_0,y_0,0)$, it is an isolated point of $\mathcal{T}\cap\mathcal{Q}$ with this property, so it is singular.
 \end{proof}

Examples \ref{ej_toro_cono_vert} and \ref{toro9-cil} illustrate Proposition \ref{multiple_components}. In Example \ref{toro9-cil}, $q_1$ is a factor of the resultant, and Example \ref{ej_toro_cono_vert} shows tangency along a component of the space curve.
When $q_1\equiv 0$, if a factor of $S_0$ is a line through the origin, then the segment of this line in ${\cal A}_{{\cal T},{\cal Q}}$ is also the projection of a meridian circle, as shown in Example \ref{cuentasjge-cil}. However, the point 2. of Proposition \ref{multiple_components} is not normally verified but it is not entirely impossible as shown in Example \ref{q10tang}.

%
%
%

\bigskip

\begin{rem}\label{coverticalidad}{\rm 
A singular point $(x_0,y_0)\not\in\Delta_{\cal T}$ on the cutcurve with $q_1(x_0,y_0)=0$, $q_1\not\equiv0$ could be either the projection of two regular points, as shown in Examples \ref{t9-elipse2} and \ref{toro-8}, or the projection of a regular point and of a singular point, as shown in Examples  \ref{cuentasjorge_arribanoabajosi} and \ref{sing_singreg}. 
In the following lines we describe how to distinguish these two situations.

Suppose that $(x_0,y_0)$ is the projection of two regular points, $(x_0,y_0,\pm\sqrt{-q_0})$ with $q_0\ne0$. Then, the projected $(x_0,y_0)$ is a double point of the cutcurve, i.e., there is a second order derivative of ${\bf S_0}$ not vanishing at $(x_0,y_0)$, since, analytically, we are projecting two non-vertical lines. On the other side, the Jacobian matrix of $T$ and $Q$ evaluated  at $(x_0,y_0,\pm\sqrt{-q_0})$ has rank 2.

However, if one of $(x_0,y_0,\pm\sqrt{-q_0})$ is singular, then the multiplicity of $(x_0,y_0)$ as a point of the cutcurve must be higher since, analytically, either at least two branches of $\mathcal{T}\cap\mathcal{Q}$ or a non-vertical cusp pass through one of the points, and at least one more branch pass through the other. In this case, Propositions
\ref{Symmetric_q1cero}
and \ref{singnotwosing} guarantee that the intersection is transversal at one of the points, and this is where the Jacobian matrix of $T$ and $Q$ has rank 2.
}
\end{rem}

\section{ Computational aspects }\label{ecuaciones}
 
In this section we analyze the structure of the different polynomials used to study the intersection curve between the torus defined by
$$
T(x,y,z):= (x^2+y^2+z^2+R^2-r^2)^2-4R^2(x^2+y^2)=z^4+p_2(x,y)z^2+p_0(x,y),
$$
and the quadric given by
\begin{multline*}
Q(x,y,z):=  a x^{2}+b y^{2}+c z^{2}+d x y +e x z +f y z +g x +h y +i z +j=c z^{2}+q_1(x,y)z+q_0(x,y)
\end{multline*}
with $c\neq 0$. In the previous sections, we assumed that c was equal to 1; however, in this one we have decided to leave the variable $c$ to homogenize the expressions.

\subsection{The polynomials $\widetilde{{\bf S}_0}$, $sres_1$ and $sres_{1,0}$}\label{equations}

The resultant of $T$ and $Q$ with respect to $z$ is a degree $8$ polynomial in $x$ and $y$ verifying
\begin{multline*}
\widetilde{{\bf S}_0}(x,y)=\left((a-c)^2+e^2\right)^2x^8+\ldots\ldots+\left((b-c)^2+f^2\right)^2y^8+\ldots\ldots\ldots\\ \ldots\ldots\ldots+\left(\left(c(R^2-r^2)-j\right)^2+i(R^2-r^2)\right)^2
\end{multline*}
Moreover, the homogeneous degree $8$ component of $\widetilde{{\bf S}_0}$ is the square of the quartic
\begin{multline*}
\left((a-c)^2+e^{2}\right) x^{4}+2 \left((a-c)d +e f \right) x^{3} y +\left(2(a-c)(b-c)
+d^{2}+e^{2}+f^{2}\right) x^{2} y^{2}
\\
+2 \left((b-c)d +e f \right) x \,y^{3}+\left((b-c)^2+f^{2}\right) y^{4}
\end{multline*}
The homogeneous degree $7$ component of $\widetilde{{\bf S}_0}$ is four times the product of this quartic and the cubic
$$\left((a -c)g +e i \right) x^{3}+\left((a -c)h +d g +f i \right) x^{2} y +\left((b -c )g +d h +e i \right) x \,y^{2}+
\left((b -c) h +f i \right) y^{3}.$$
The dependence of $R$ and $r$ in $\widetilde{{\bf S}_0}$ appears for first time in its homogeneous degree $6$ component. Finally, the homogeneous degree $1$ component of $\widetilde{{\bf S}_0}$ is

\begin{multline*}
-4\left(\left(c(R^2-r^2)-j\right)^2+i(R^2-r^2)\right)\left(
\left((cg-ei)(R^2-r^2)-gj\right)x+
\left((ch-fi)(R^2-r^2)-hj\right)y\right)
\end{multline*}

This implies that when $b=c$ and $f=0$, the degree of $\widetilde{{\bf S}_0}$ with respect to $y$ is less than or equal to $4$. Similarly, when $a=c$ and $e=0$, the degree of $\widetilde{{\bf S}_0}$ with respect to $x$ is less than or equal to $4$.

Furthermore, when $e=f=i=0$, ie $q_1(x,y)\equiv 0$, $\widetilde{{\bf S}_0}$ is the square of the polynomial $sres_{1,0}/c$ given by 
\begin{equation*}
\begin{array}{l}
\left(\left(a -c \right) x^{2}+d x y +\left(b -c \right) y^{2}\right)^{2}
+2\left(\left(a -c \right) x^{2}+d x y +\left(b -c \right) y^{2}\right)\left(gx+hy\right)-\\
\qquad  -\left(2a(c(R^2-r^2)-j)+2c^2(R^2+r^2)+2cj-g^2\right)x^2-2\left(d(c(R^2-r^2)-j)-gh \right)xy-\\
\qquad -\left(2b(c(R^2-r^2)-j)+2c^2(R^2+r^2)+2cj-h^2\right)y^2
-2 \left(c(R^{2}-r^{2})-j \right) \left(g x +h y \right)+\\\hfill+\left(c(R^{2}-r^{2})-j \right)^{2}
\end{array}
\end{equation*}

Actually, if $q_1(x_0,y_0)=0$ and $(x_0,y_0)$ is on the cutcurve, then $x_0$ is a root of the quartic univariate polynomial
\begin{equation*}
\begin{array}{l}
\left((a-c) f^{2}+(b-c) e^{2}-d e f \right)^{2} x^{4} +2 \left((a-c) f^{2}+(b-c) e^{2}-d e f \right)\left(2 b e i -2 c e i -d f i -e f h +f^{2} g \right)  x^{3} +\ldots\\
\qquad\ldots+\left(R^{2} c f^{2}-c f^{2} r^{2}-2 R c f i -b i^{2}+c i^{2}-f^{2} j +f h i \right) \left(R^{2} c f^{2}-c f^{2} r^{2}+2 R c f i -b i^{2}+c i^{2}-f^{2} j +f h i \right)
\end{array}
\end{equation*}

\subsection{Determining the intersection points of the cutcurve with the silhouette curves}\label{intersectwithsilhouette} 

We show here how to determine the intersection points of the cutcurve with the silhouette curves. We assume again $c=1$. We begin by considering the case where the three curves intersect. The following two lemmas establish new relationship between  the polynomials $\widetilde{{\bf S}_0}$, $\Delta _{{\cal T}}(x,y)$ and $\Delta _{{\cal Q}}(x,y)$ and they will be used to determine the intersection of the cutcurve with the silhouette curves.
Recall that $\Delta _{{\cal T}}=p_0$ and $\Delta _{{\cal Q}}=q_1^2-4 q_0 $.

\begin{lem}\label{cutcurveformula2}
$$
\widetilde{{\bf S}_0}= Q_1 \Delta _{{\cal T}}+Q_2  \Delta _{{\cal Q}} +q_0^2 (p_2  + q_0)^2,
$$
with
$$
Q_1=p_0- 2p_2q_0 + p_2q_1^2 + 2q_0^2 - 4q_0q_1^2 + q_1^4 \, \text{ and } \,
Q_2=p_2 q_0^2;
 $$
\end{lem}

\begin{lem}\label{cutcurveformula}
$$
\widetilde{{\bf S}_0}= R_1\, q_1^2+R_2\, q_0 +  p_0^2,
$$
with
$$R_1=p_0 (p_2   - 4  q_0   +  q_1^2) + p_2 q_0^2 \,
\text{ and } \,
R_2=(q_0-p_2) (2 p_0   - p_2 q_0   + q_0^2).$$
\end{lem}

As a direct consequence of these two lemmas, the intersection of the cutcurve with the silhouette curves can be determined by intersecting a line, a conic and the product of two circles centered at the origin.

\begin{propst} The systems of polynomial equations
$$\widetilde{{\bf S}_0}(x,y)=0, \quad \Delta _{{\cal T}}(x,y)=0 , \quad \Delta _{{\cal Q}}(x,y)=0$$
and $$q_1(x,y)=0, \quad q_0(x,y)=0, \quad p_0(x,y)=0$$
have exactly the same solutions.
\end{propst}

Thus, by Corollary \ref{q10S}, next result shows a particular case where the point on the cutcurve and on the silhouette curves is singular.

\begin{cor}
If a real point $(x_0,y_0)$ satisfies ${\bf S}_0(x_0,y_0)=\Delta _{{\cal T}}(x_0,y_0)=\Delta _{{\cal Q}}(x_0,y_0)=0$, and $\widetilde{{\bf S}_0}(x,y)$ is squarefree, then $(x_0,y_0)$ is a singular point of the cutcurve.
\end{cor}

Observe that the intersection points of the cutcurve with the silhouettes may not be singularities if $\widetilde{{\bf S}_0}(x,y)$ is not squarefree, as shown in Example \ref{curvareal}.
Finally, the following corollary illustrates that if $q_1\equiv 0$, it suffices to find the intersection points of a conic with the product of two circles centered at the origin to determine the points on the three curves.


\begin{cor}  If $q_1\equiv 0$, then the systems of polynomial equations
$$\widetilde{{\bf S}_0}(x,y)=0, \, \Delta _{{\cal T}}(x,y)=0 , \, \Delta _{{\cal Q}}(x,y)=0 $$ and
$$ \Delta _{{\cal T}}(x,y)=0 ,  \,\Delta _{\cal Q}(x,y)=0$$
have exactly the same solutions.
\end{cor}

The next proposition shows how to compute the intersection of the cutcurve with the silhouette curve $\Delta _{{\cal T}}=0$. It is reduced to intersect a conic with the product of two circles centered at the origin.

\begin{propst} If $(x_0, y_0)\in \R^2$ then
$$
\widetilde{{\bf S}_0}(x_0,y_0)=\Delta _{{\cal T}}(x_0,y_0)=0  ,\,  \Delta _{{\cal Q}}(x_0,y_0)\geq 0   \Longleftrightarrow q_0(x_0, y_0)=0,\, \Delta _{{\cal T}}(x_0,y_0)=0\, .
$$
\end{propst}
\begin{proof}
First suppose that $\widetilde{{\bf S}_0}(x,y)=\Delta _{{\cal T}}(x_0,y_0)=0  $ and  $\Delta _{{\cal Q}}(x_0,y_0)\geq 0$. Then, the point $(x_0, y_0)$ in on  the cutcurve such that there exists at least a real $z_0$  with $(x_0, y_0,z_0) \in {\cal T} \cap {\cal Q}$. Since  $\Delta _{{\cal T}}(x_0,y_0)=0 $, it follows that $z_0=0$, which implies $q_0(x_0,y_0)=0$.

Next, assume that  $q_0(x_0, y_0)=0,\Delta _{{\cal T}}(x_0,y_0)=0$. Then, $(x_0, y_0, 0) \in {\cal T} \cap {\cal Q}$, which implies $\widetilde{{\bf S}_0}(x_0,y_0)=0$.

\end{proof}

Finally, we show how to compute the intersection of the cutcurve with the silhouette curve $\Delta _{{\cal Q}}=0$. It is reduced to intersect a conic with a quartic (note that the degree of  $\widetilde{{\bf S}_0}(x,y)$ is smaller than or equal to $8$).

\begin{propst} If $(x_0, y_0)\in \R^2$ then
$$
\widetilde{{\bf S}_0}(x_0,y_0)=\Delta _{{\cal Q}}(x_0,y_0)= 0  ,  \Delta _{{\cal T}}(x_0,y_0)\leq 0   \Longleftrightarrow $$  $$\Longleftrightarrow16\, p_0(x_0,y_0)  + 4\, p_2(x_0,y_0) q_1(x_0,y_0)^2  + q_1(x_0,y_0)^4=0, \Delta _{{\cal Q}}(x_0,y_0)=0\, .
$$
\end{propst}

\begin{proof}
First suppose that $\widetilde{{\bf S}_0}(x_0,y_0)=\Delta _{{\cal Q}}(x_0,y_0)=0  $ and  $\Delta _{{\cal T}}(x_0,y_0)\leq 0$. Then, by Equation (\ref{fact}), we have
$$
\widetilde{{\bf S}_0}(x_0,y_0)={\big(16\, p_0(x_0,y_0)  + 4\, p_2(x_0,y_0) q_1(x_0,y_0)^2 + q_1(x_0,y_0)^4\big)^2}/{2^8 },
$$
and  $ 16\, p_0(x_0,y_0)  + 4\, p_2(x_0,y_0 )q_1(x_0,y_0)^2  + q_1(x_0,y_0)^4 =0$.

Next, assume that   $ 16\, p_0(x_0,y_0)  + 4\, p_2(x_0,y_0)q_1(x_0,y_0)^2  + q_1(x_0,y_0)^4 =\Delta _{{\cal Q}}(x_0,y_0)=0$.
Then $\widetilde{{\bf S}_0}(x_0,y_0)=0$ so that there exists $z_0$ with ${\cal T}(x_0,y_0,z_0)={\cal Q}(x_0,y_0,z_0)=0$. In fact, $z_0=-q_1(x_0,y_0)/2  \in \mathbb{R}$, which implies $(x_0,y_0)$ is on the cutcurve and thus $p_0(x_0,y_0)\leq 0$.
\end{proof}

\subsection{The intersection of the torus $\mathcal{T}$ with the quadric $$\frac{x^{2}}{A}+\frac{y^{2}}{B}+\frac{z^{2}}{C}=1$$}\label{canonical}
In this case, $\widetilde{{\bf S}_0}(x,y)$ is the square of the polynomial
\begin{multline*}
{\bf S}_0(x,y)=
\left( B(A -C) x^{2}+A (B -C ) y^{2}  \right)^2
+2 A B^{2} \left(-A R^{2}-A r^{2}-C R^{2}+C r^{2}+C A -C^{2}\right) x^{2}+\\
+2 A^{2} B \left(-B R^{2}-B r^{2}-C R^{2}+C r^{2}+B C -C^{2}\right) y^{2}
+A^{2} B^{2} \left(R^{2}-r^{2}+C \right)^{2}
\end{multline*}

Replacing $x^2=X$ and $y^2=Y$ we obtain a conic with matrix  
\[
{\cal C} =
\left(\begin{array}{ccc}
{\cal C}_{11} & {\cal C}_{12} & {\cal C}_{13} \\
{\cal C}_{12} & {\cal C}_{22} & {\cal C}_{23} \\
{\cal C}_{13} & {\cal C}_{23} & {\cal C}_{33}
\end{array}\right)
\]
with ${\cal C}_{11} = B^{2} \left(A -C \right)^{2}$, ${\cal C}_{12} = A B \left(B -C \right) \left(A -C \right)$,
${\cal C}_{13} = A B^{2} \left(-A R^{2}-A r^{2}-C R^{2}+C r^{2}+C A -C^{2}\right)$, ${\cal C}_{22} = A^{2} \left(B -C \right)^{2}$, ${\cal C}_{23} = A^{2} B \left(-B R^{2}-B r^{2}-C R^{2}+C r^{2}+B C -C^{2}\right)$ and $ {\cal C}_{33} = A^{2} B^{2} (R^{2}-r^{2}+C )^{2}$.

Since $\det({\cal C}[1..2,1..2])=0$ and $$\det({\cal C})=-4 A^{4} B^{4} C^{2} R^{4} \left(A -B \right)^{2}\leq 0$$
we conclude that, when $A\neq B$, the conic given by ${\cal C}$ is a parabola and,  when $A=B$, it is the product the two parallel lines given by
$$(A-C)(X+Y)+\alpha=0\qquad (A-C)(X+Y)+\beta=0$$
with
\begin{multline*}
\alpha=\frac{2 A \left(-A \,R^{2}-A r^{2}-C R^{2}+C r^{2}+C A -C^{2}\right)}{A -C}+\frac{2 A R \sqrt{A \left(A \,r^{2}+C R^{2}-C r^{2}-C A +C^{2}\right)}}{A -C}\\
\beta=\frac{2 A \left(-A R^{2}-A r^{2}-C R^{2}+C r^{2}+C A -C^{2}\right)}{A -C}-\frac{2 A R \sqrt{A \left(A r^{2}+C R^{2}-C r^{2}-C A +C^{2}\right)}}{A -C}
\end{multline*}
This implies that the cutcurve is the image of the conic ${\cal C}$ under the map
$$(X,Y)\mapsto (\pm\sqrt{X},\pm\sqrt{Y)}$$
providing a piecewise parametric representation (not rational) of the cutcurve and of the searched intersection curve between $\mathcal{T}$  and the considered quadric.

\bigskip
\section{Examples}\label{secej}
The following examples illustrate the different results that appear in this work.

\begin{examp}\label{esfera} {\rm Consider the torus with $r=2,R=7$ and the sphere $ {\cal Q}: x^2 + y^2 + z^2 - 6^2=0$. We have $q_1\equiv 0$ and $\widetilde{S}_0=(196x^2 + 196y^2 - 6561)^2.$ In this example, the circle defined by the squarefree part of the resultant defines the cutcurve.
Furthermore, $\cal C_\mathbb R$ is described by two parallel circles (see \cite{bottema}; also called profile circles in \cite{Kim} and \cite{Li2004AlgebraicAF}) whose projections overlap. Obviously, $\mathcal{C}_\mathbb{R}$ is symmetric wrt the plane $z=0$ and every point  $(x_0,y_0)$ of the cutcurve is projection of $(x_0,y_0,\pm\sqrt{-q_0(x_0,y_0)})$. See Figure \ref{ex_esfera}. There are no singularities.


\begin{figure}[H]
\centering
 \fbox{\includegraphics[scale=0.25]{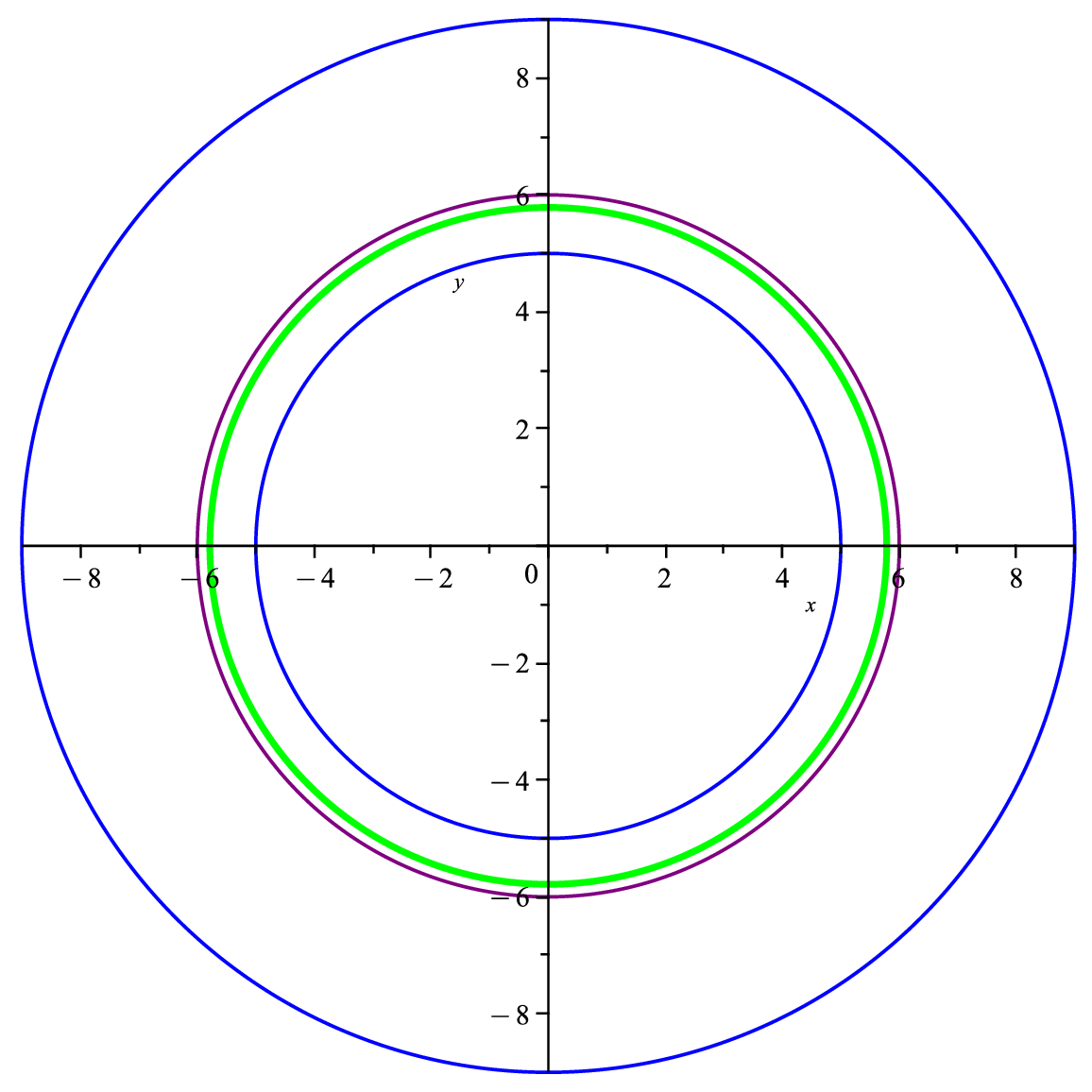}}\qquad
 \fbox{\includegraphics[scale=0.3]{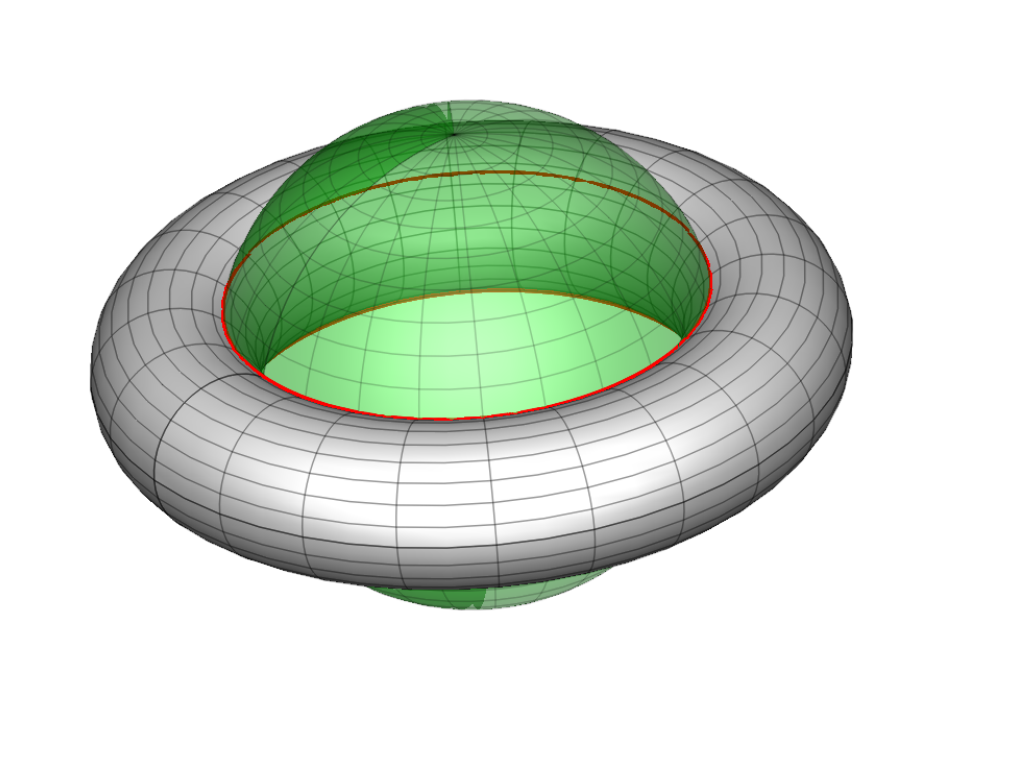}}
\caption{\label{ex_esfera} Left: $\Delta_{\cal T}$ (blue), $\Delta_{\cal Q}$ (purple) and the cutcurve (green). Right: The intersection curve defined by two parallel circles (red).}
\end{figure}
}\end{examp}

\begin{examp}\label{ej_toro_cono_vert} {\rm 
Consider the torus with $r=5$ and $R=6$, and a circular cone with the same axis as the torus (z-axis),
$ Q = (z - \frac{7}{4})^2 - \frac{9}{16} x^2 - \frac{9}{16} y^2 .$ We have that $q_1$ is a constant and the resultant is neither squarefree nor a perfect square,
$$\widetilde{{\bf S}_0}=\frac{390625}{65536} \left((x^2+y^2)^2 -\frac{43506}{625}(x^2+y^2) +81 \right) (x^2+y^2-9)^2 .$$
Here the discriminant of the cone is never negative over $\R$ and the cutcurve is directly defined by the squarefree part of the resultant, that is, by three circles. Every circle is projection of a parallel, but it should be noted that
the cone and a torus are tangent along just one, the parallel circle of radius $3$.
See Figure \ref{ex_toro_cono_1}. This example illustrates Propositions \ref{resultant_es_cutc} and  \ref{multiple_components} and Theorem \ref{lma:mult_comp}. There are no singularities.

\begin{figure}[H]
\centering
 \fbox{\includegraphics[scale=0.25]{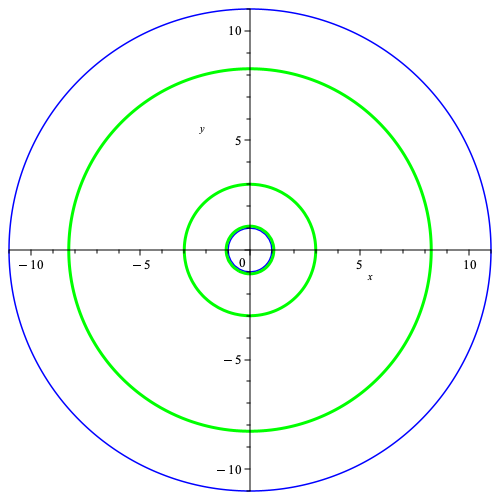}}\qquad
 \fbox{\includegraphics[scale=0.3]{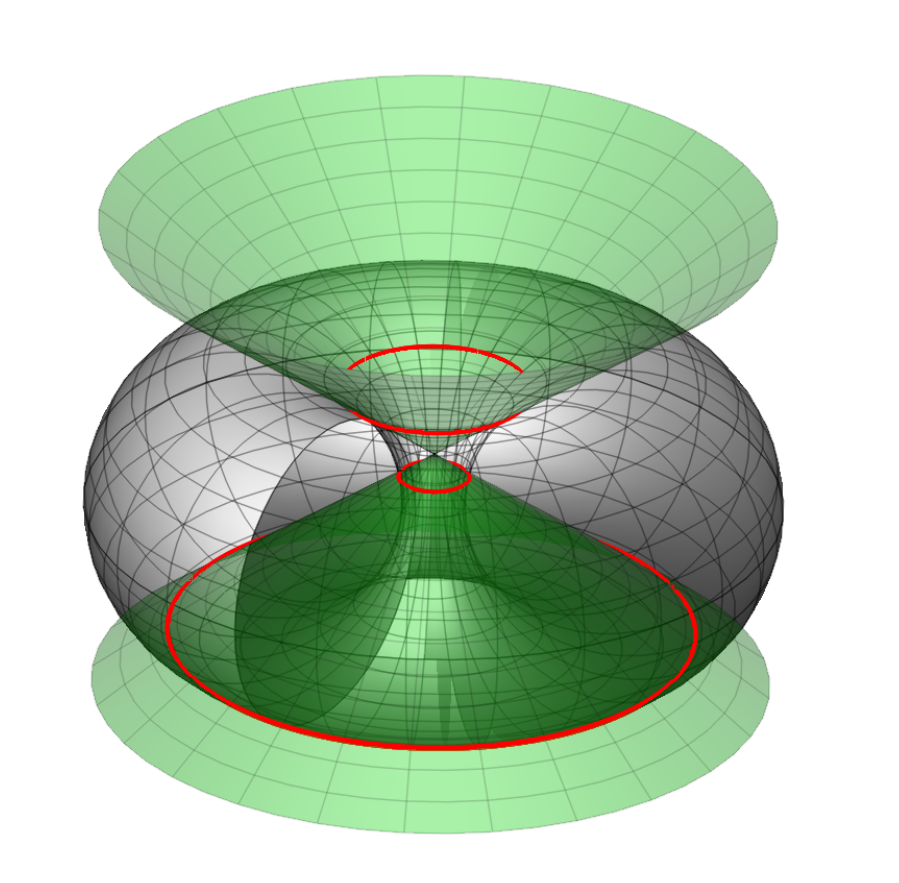}}
\caption{\label{ex_toro_cono_1} Left: $\Delta_{\cal T}$ (blue) and the cutcurve (green). Right: The intersection curve defined by three parallel circles (red).}
\end{figure}
}
\end{examp}

\begin{examp}\label{curvareal} {\rm 
Consider the torus with $r=2$ and $R=7$, and  ${\cal Q}$ is a hyperboloid of one sheet, defined by
\begin{eqnarray*}
Q(x,y,z)&=&\frac{x^2}{25}+\frac{y^2}{100}-\frac{z^2}{16}-1.
\end{eqnarray*}
Here, $q_1\equiv 0$ and 
\[
\widetilde{{\bf S}_0}= \frac{1}{160000^2}
\left(1681 x^{4}+2378 x^{2} y^{2}+841 y^{4}-63050 x^{2}-80450 y^{2}+525625\right)^{2}\,.
\]
The curve defined by ${\bf S}_0(x,y)$ is not completely contained inside the region  ${\cal A}_{{\cal T},{\cal Q}}$ (see Figure \ref{ex1}). In fact, the cutcurve is given by two arcs of the resultant curve. There are two singular points on $ \mathcal{C}_\mathbb{R}$ with $z_0=0$, whose projection are smooth points of the cutcurve, tangential intersection points of  ${\bf S}_0(x,y)$, $\Delta_{\cal Q}$ and $\Delta_{\cal T}$. This example illustrates Theorem \ref{Singq_1Siemprecero_enCorona}.

\begin{figure}[H]
\centering
 \fbox{\includegraphics[scale=0.25]{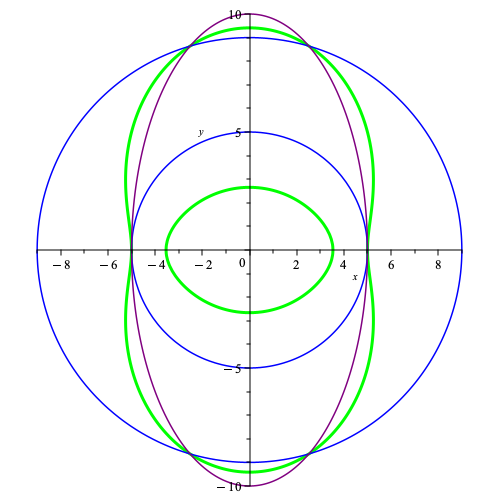}}\qquad
 \fbox{\includegraphics[scale=0.30]{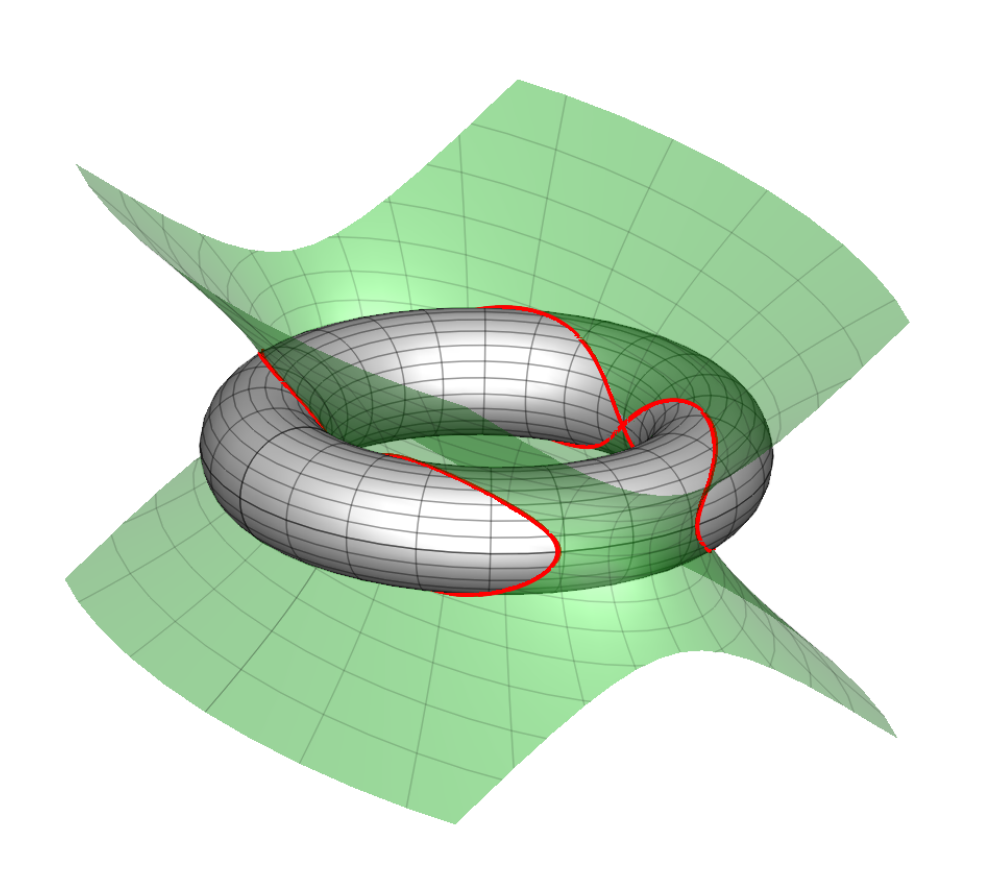}}
\caption{\label{ex1} Left: $\Delta_{\cal T}$ (blue), $\Delta_{\cal Q}$ (purple) and the curve defined by ${\bf S}_0(x,y)$ (green) for Example \ref{curvareal}. Right: The intersection curve (red).}
\end{figure}
}
\end{examp}

\begin{examp}\label{t9-elipse2} {\rm 
Consider the torus with $r=\frac{3}{2}$ and $R=3$, and the ellipsoid ${\cal Q}$, defined by
$Q(x,y,z)=z^2 - \frac{3}{5} xz +\frac{3}{4} x^2+\frac{3}{25} y^2 - 3 .$
Here $q_1=- \frac{3}{5} x$ and the resultant is squarefree. Hence, the cutcurve is directly defined by the resultant polynomial.

The cutcurve has four singularities; all of them are in the line $\{q_1=0 \}$ and they are all projection of regular points. 
Each of them is lifted to a pair of symmetric regular points on $ \mathcal{C}_\mathbb{R}$. This example illustrates Proposition \ref{resultant_es_cutc}, Corollary \ref{q10S} and Remark \ref{coverticalidad}. See Figure \ref{ex3}.
 

\begin{figure}[H]
\centering
 \fbox{\includegraphics[scale=0.25]{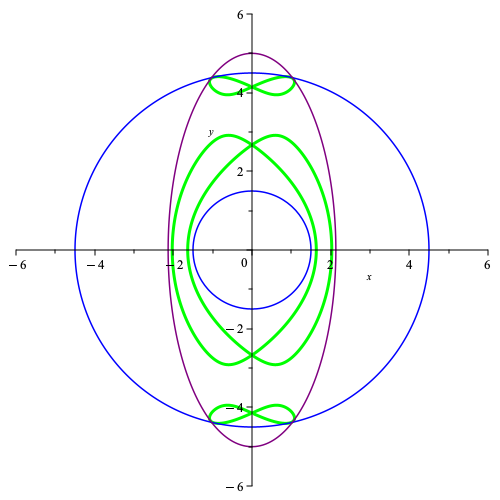}}\qquad
 \fbox{\includegraphics[scale=0.35]{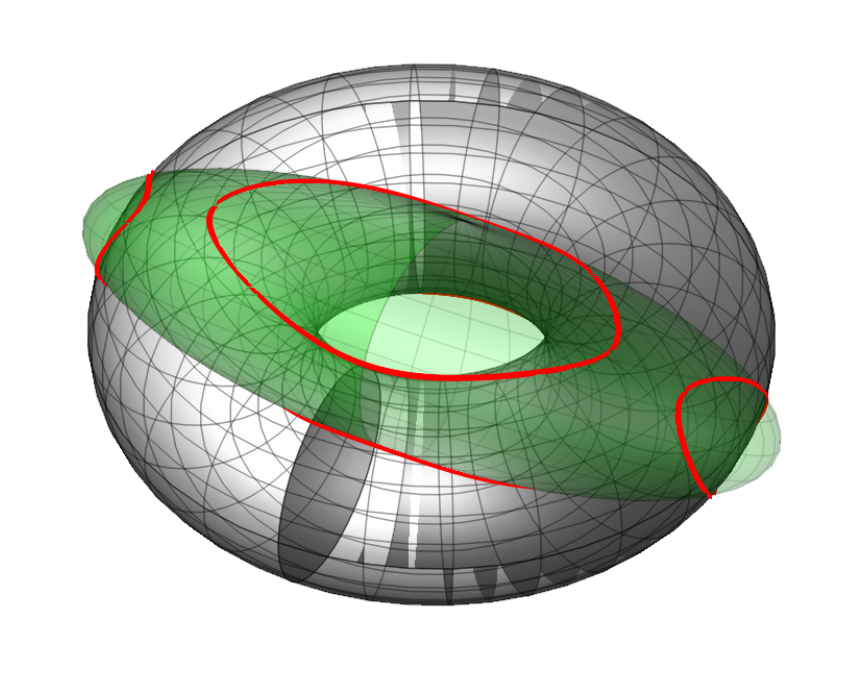}}
\caption{\label{ex3} Left: $\Delta_{\cal T}$ (blue), $\Delta_{\cal Q}$ (purple) and the cutcurve (green).
Right: The intersection curve (red). For example \ref{t9-elipse2}}
\end{figure}
 }\end{examp}

\begin{examp}\label{toro-8} {\rm 
Consider the torus with $r=1$ and $R=3$, and the ellipsoid defined by
\begin{eqnarray*}
Q(x,y,z)&=&z^{2}+\left(\frac{10 y}{21}-\frac{5}{42}+\frac{10 x}{21}\right) z +\frac{5 x^{2}}{3}+\frac{5 y^{2}}{12}-\frac{5 y}{12}-\frac{155}{48}+\frac{10 x}{3}\, .
\end{eqnarray*}

The cutcurve is directly defined by the resultant, which is squarefree, and has two singularities, see Figure \ref{ex4}. One of them is on the line $q_1=0$ and it is projection of two regular points. 
However, the other one is the projection of a singular point of the intersection curve with $z_0\neq 0$. 
This example illustrates Corollary \ref{q10S}, Theorem \ref{singarribaabajo}  and Proposition \ref{Singq_1nocero_noCorona}.

\begin{figure}[H]
\centering
 \fbox{\includegraphics[scale=0.25]{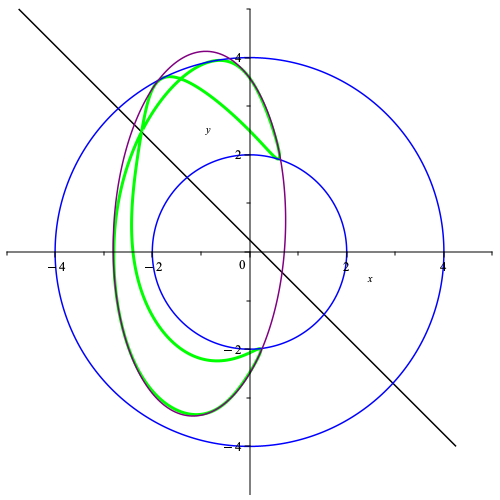}}\qquad
  \fbox{\includegraphics[scale=0.35]{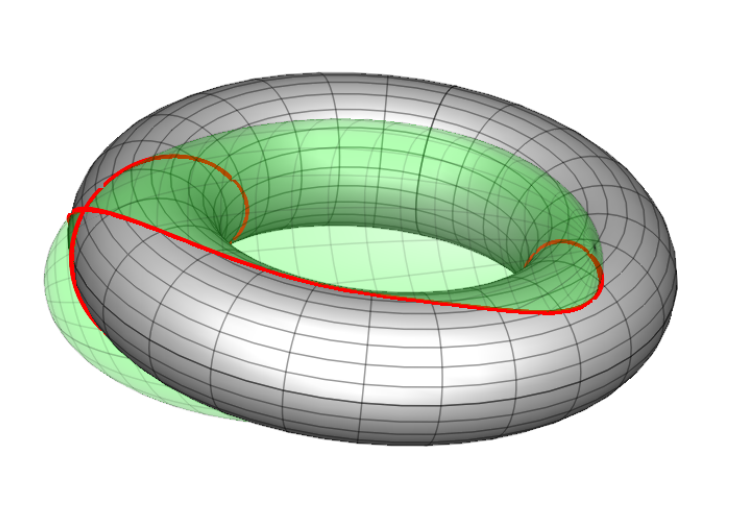}}
\caption{\label{ex4} Left: $\Delta_{\cal T}$ (blue), $\Delta_{\cal Q}$ (purple), the cutcurve (green) and the line $q_1=0$ (black) for Example \ref{toro-8}.
Right: The intersection curve (red).}
\end{figure}

}\end{examp}

\begin{examp}\label{YVC} {\rm 
Consider the torus with $r=1$ and $R=3$, and the sphere defined by
\begin{eqnarray*}
Q(x,y,z)&=&\left(z -2\right)^{2}+\left(x -\frac{\sqrt{6}}{2}\right)^{2}+y^{2}-\frac{27}{2}.
\end{eqnarray*}
Observe that $q_1$ is a constant. The cutcurve is directly defined by the resultant, which is squarefree, and has two singularities. 
Both are  projections of singular points of the intersection curve, which is defined by two Yvone-Villarceau circles  (see \cite{Kim}),  with $z_0\neq 0$. This example illustrates Theorem \ref{singarribaabajo} and Proposition \ref{Singq_1nocero_noCorona}. See Figure \ref{exyvc}.

\begin{figure}[H]
\centering
 \fbox{\includegraphics[scale=0.25]{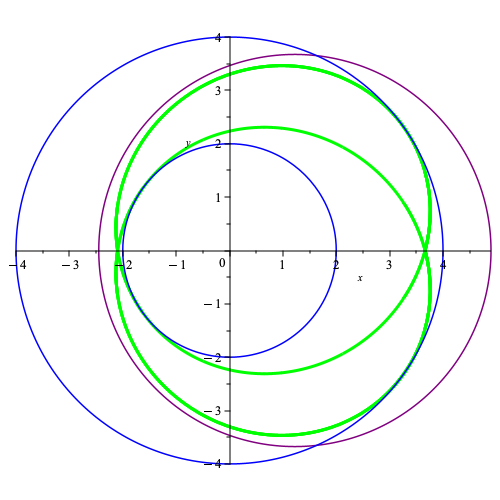}}\qquad
  \fbox{\includegraphics[scale=0.3]{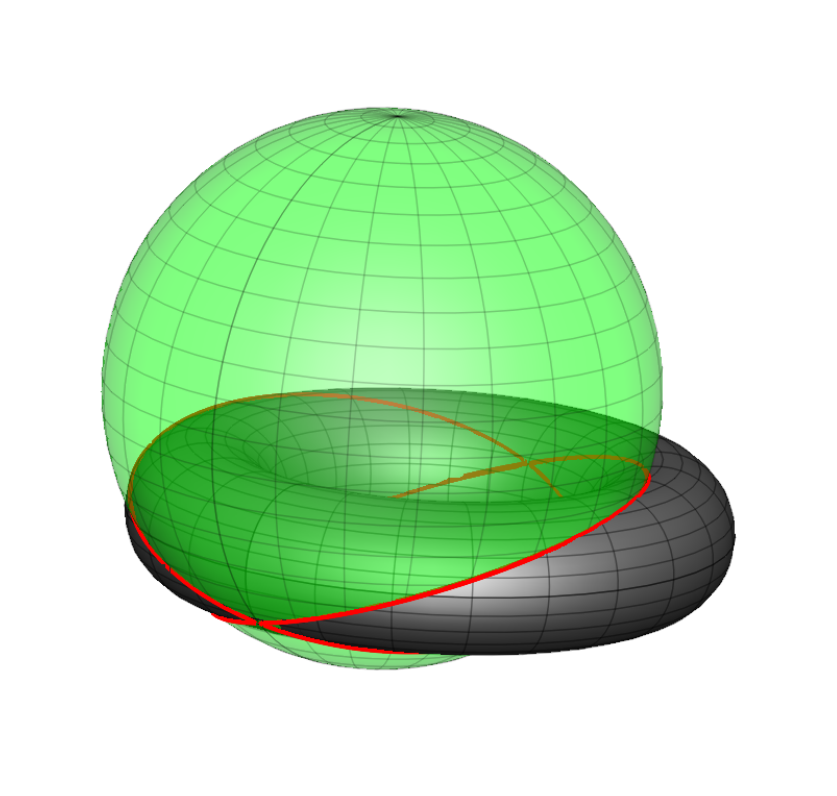}}
\caption{\label{exyvc} Left: $\Delta_{\cal T}$ (blue), $\Delta_{\cal Q}$ (purple) and the cutcurve (green) .
Right: Intersection curve (red) for Example \ref{YVC}.}
\end{figure}

}
\end{examp}

%
%
%
%
%

\begin{examp}\label{varios-5} {\rm 
Consider the torus with $r=1, R=3$ and the ellipsoid defined by
$$Q(x,y,z)=z^{2}+\left(-x -y +2\right) z +3 x^{2}-20+3 y^{2}+4 x = 0 .
$$
Here $q_1\not\equiv 0$ and the resultant is also squarefree. The cutcurve has two singular points and both of them satisfy $q_1=0$. One is the point $(2,0)$, where both silhouettes meet tangentially. This point lifts to the singular point $(2,0,0)$. 
However, the other singular point lifts to two regular points in the intersection curve. 
This example illustrates Corollary \ref{q10S}, Theorem \ref{cortangent2} and Corollary \ref{porfin}. See Figure \ref{toro-elip-5}.
 
 \begin{figure}[H]
\centering
 \fbox{\includegraphics[scale=0.25]{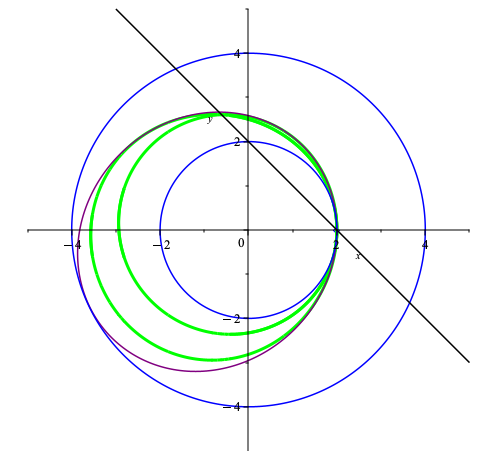}} \hspace*{0.5cm}
 \fbox{\includegraphics[scale=0.3]{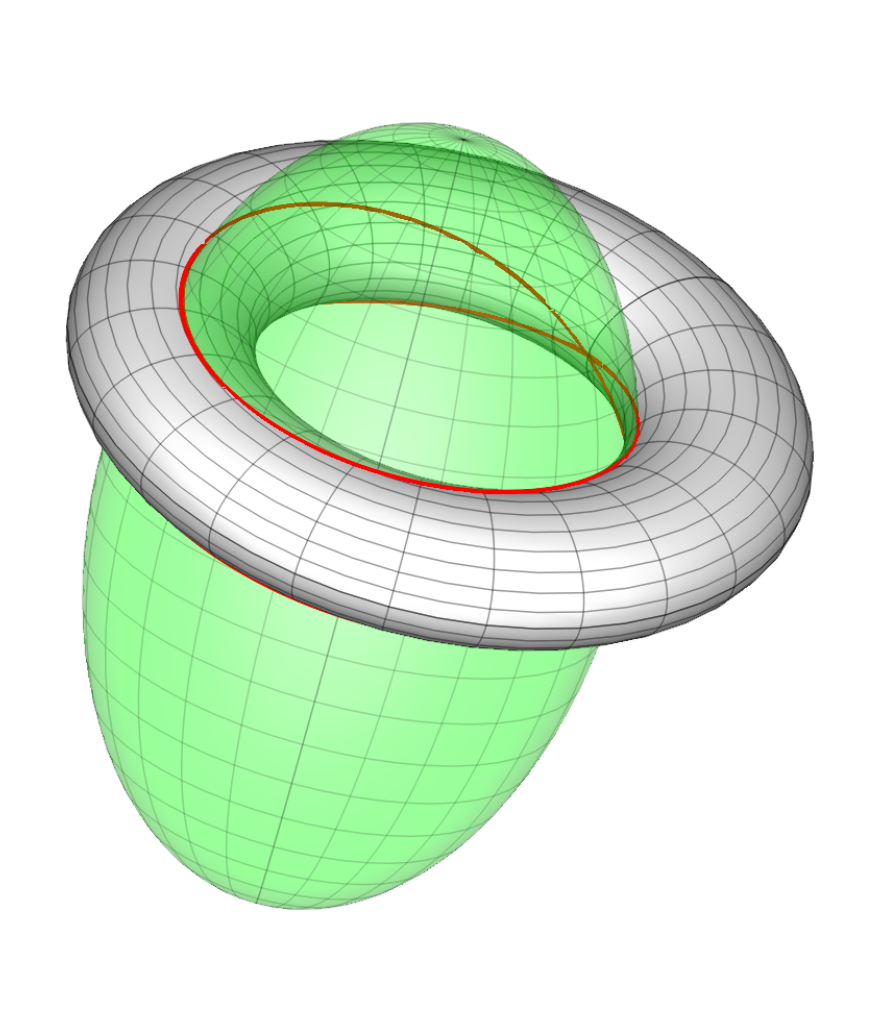}}
\caption{\label{toro-elip-5}
Left: $\Delta_{\cal T}$ (blue), $\Delta_{\cal Q}$ (purple), the line $q_1=0$ (black) and the cutcurve (green). Right: The intersection curve (red) for Example \ref{varios-5}. }
\end{figure}
}
\end{examp}


\begin{examp}\label{toro9-cil} {\rm 
Consider the torus with $r=1$ and $R=4$, and the elliptic cylinder
$$ Q(x,y,z)=(x-4)^2 + (z+y)^2 -1 =0 .$$
Then $ \widetilde{{\bf S}_0}= 16 \left(y^{6}+\left(2 x^{2}-2\right) y^{4}+\left(x^{4}+62 x^{2}-512 x +961\right) y^{2}+64 x^{4}-512 x^{3}+960 x^{2}\right) y^{2} $.

Observe that $q_1$ is equal to $2 y$, so that $q_1^2$ divides the resultant. Thus, by Proposition \ref{multiple_components}, there is a segment on the cutcurve which is the projection of a meridian circle of $\cal T \cap \cal Q$, the segment that joins $(3, 0)$ to $(5, 0)$. Moreover, these points are singularities of the cutcurve because they belong to the two components of the resultant curve. This implies that the intersection curve has two components as well, and their intersection points correspond precisely to the singular points that project onto $(3, 0)$ and $(5, 0)$, case. 
See Figure \ref{toro-cyl-elip}.
\begin{figure}[H]
\centering
 \fbox{\includegraphics[scale=0.25]{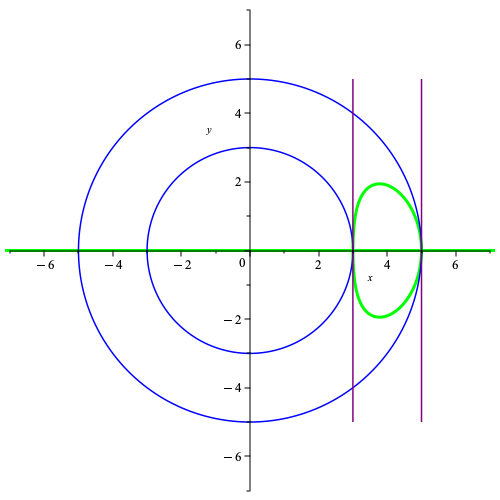}} \hspace*{0.1cm}
 \fbox{\includegraphics[scale=0.25]{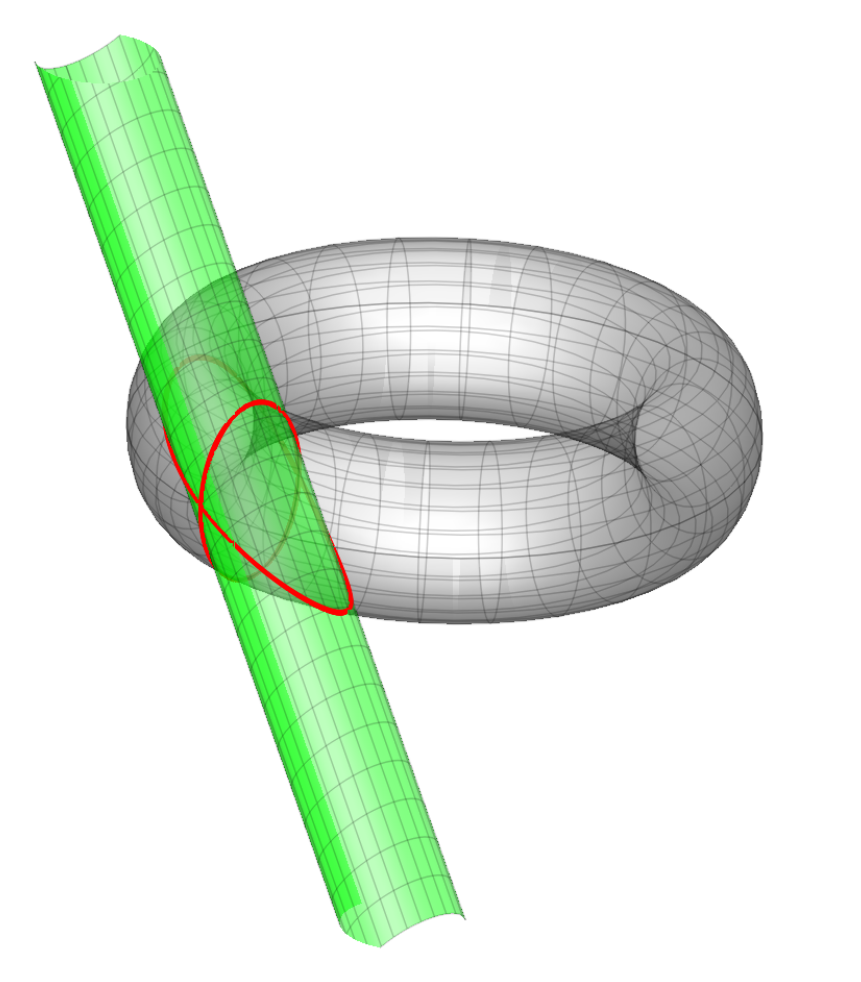}}\;
 \fbox{\includegraphics[scale=0.75]{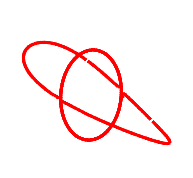}}
\caption{\label{toro-cyl-elip}
Left: $\Delta_{\cal T}$ (blue), $\Delta_{\cal Q}$ (purple) and the cutcurve (green). Center/Right: Intersection curve (red) for Example \ref{toro9-cil}. }
\end{figure}
}
\end{examp}

\begin{examp}\label{cuentasjge-cil}
 {\rm 
Consider the torus with $r=1$ and $R=4$, and the cylinder
$$  Q(x,y,z)=z^2 + (x - 4)^2-1=0$$
Here $q_1\equiv 0$, $ \widetilde{{\bf S}_0}= (y^2(y^2 + 16x - 64))^2 $ and the cutcurve is defined by two arcs contained in ${\cal A}_{{\cal T},{\cal Q}}$. In Figure \ref{cuentasjge-cil-drw}, observe that the segment which joins $(3, 0)$ to $(5, 0)$ is also the projection of a meridian circle of $\cal T \cap \cal Q$ along which the torus and the cylinder are tangent.
Observe that hypotheses of Theorem \ref{Singq_1Siemprecero_enCorona} are not verified. On the other hand, the cutcurve has a singularity which corresponds to the projections of two singularities on the space curve, with $z_0\neq 0$. 
This example illustrates Theorems \ref{singarribaabajo}, \ref{cortangent2} and \ref{lma:mult_comp}.

\begin{figure}[H]
\centering
 \fbox{\includegraphics[scale=0.25]{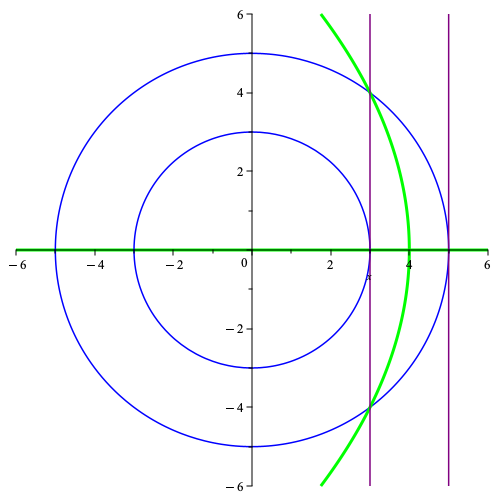}} \hspace*{0.1cm}
 \fbox{\includegraphics[scale=0.35]{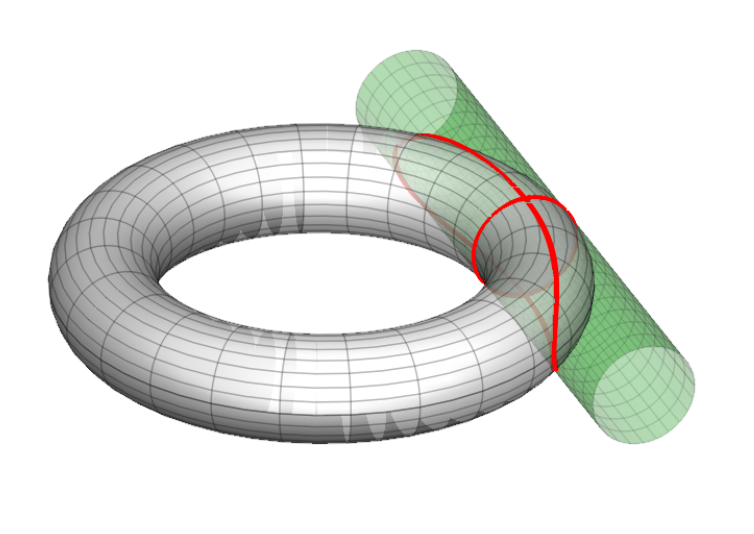}}
\caption{\label{cuentasjge-cil-drw}
Left: $\Delta_{\cal T}$ (blue), $\Delta_{\cal Q}$ (purple) and the cutcurve (green). Right: Intersection curve (red) for Example \ref{cuentasjge-cil}. }
\end{figure}
}
\end{examp}

\begin{examp}\label{toro9-2sheet}
 {\rm 
Consider the torus with $r=2$ and $R=4$, and the two sheet hyperboloid
$$ Q(x,y,z)=z^{2}+\frac{9}{4} y^{2} -\frac{9}{16} x^{2}-\frac{9}{4} x +\frac{27}{4}=0 .$$
Then $q_1 \equiv 0$ and $ \widetilde{{\bf S}_0}= {\bf S}_0^2 $. The cutcurve is defined by an arc and a single point of the resultant curve inside the region $\cal A _{\cal T,\cal Q}$.
In Figure \ref{toro-2sheet}, we can see that the point $(2, 0)$ on the cutcurve is tangent to the cutcurve curve and both silhouette curves, and lifts to the self-intersection point of the intersection curve $(2, 0, 0),$ in the plane $\{z=0\}$.
On the other hand, the point $(-6, 0)$, where the cutcurve is tangent to both silhouette curves as well, corresponds to the isolated point $(-6, 0, 0)$ on the intersection curve. 
This example illustrates Theorem  \ref{Singq_1Siemprecero_enCorona}. 

\begin{figure}[H]
\centering
 \fbox{\includegraphics[scale=0.25]{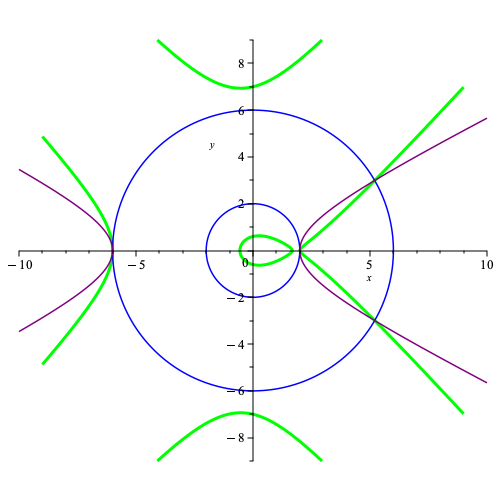}}
\hspace*{0.2cm}
 \fbox{\includegraphics[scale=0.35]{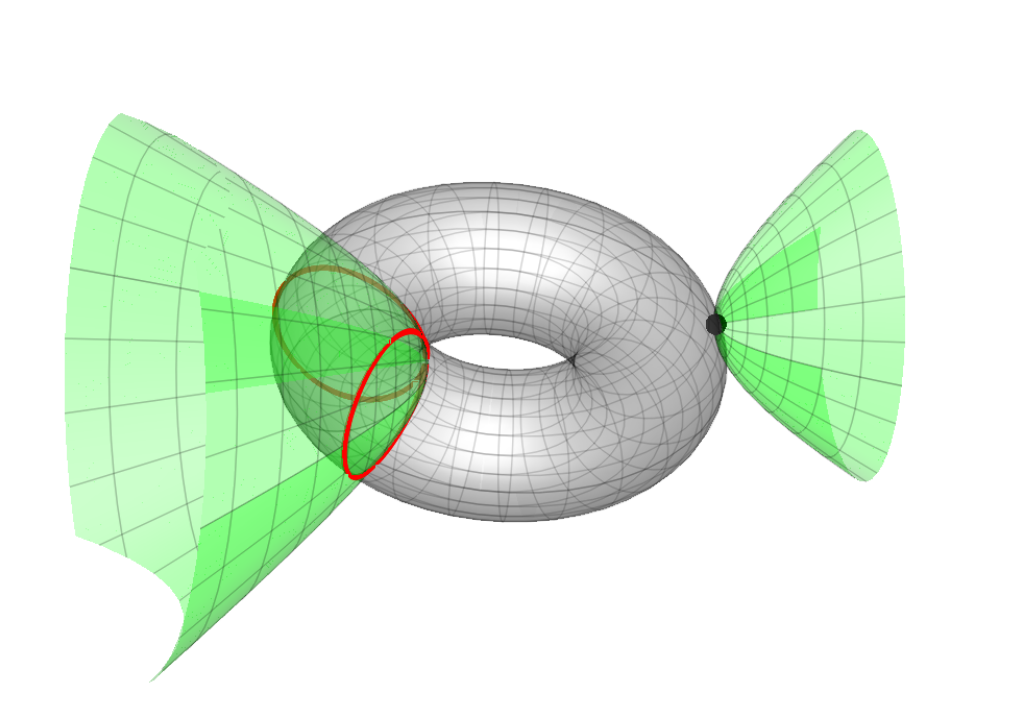}}
\caption{\label{toro-2sheet}
Left: $\Delta_{\cal T}$ (blue), $\Delta_{\cal Q}$ (purple) and the cutcurve (green). Right: Intersection curve (red) for Example \ref{toro9-2sheet}. }
\end{figure}
}
\end{examp}

\begin{examp}\label{cuentasjorge_arribanoabajosi} {\rm 
Consider the torus with $r=1$ and $R=4$, and the cone
$$ Q(x,y,z)=z^2 +yz +  (x - 4)^2 - 1 + y=0 .$$
Then $q_1 =y$ and $ \widetilde{{\bf S}_0} =y^2(x^4y^2 + 4x^2y^4 + \cdots+ 6144y^2) $. Observe that the cutcurve has two components: the segment of the line $y=0$ in ${\cal A}_{{\cal T},{\cal Q}}$, which is the projection of a meridian circle, and the one defined by the other factor of the resultant.

In Figure \ref{fig_cuentasjorge_arribanoabajosi}, observe that the point $(4, 0)$ on the cutcurve is singular, intersection of the two components of the cutcurve. It lifts to two points, the point $(4,0,1)$, which is not singular, and the point $(4,0,-1)$, which is a singular point of the intersection curve. 
This example illustrates Corollary \ref{singnotwosing}, Theorems \ref{singarribaabajo} and \ref {lma:mult_comp}, and Proposition \ref{multiple_components}.
\begin{figure}[H]
\centering
 \fbox{\includegraphics[scale=0.25]{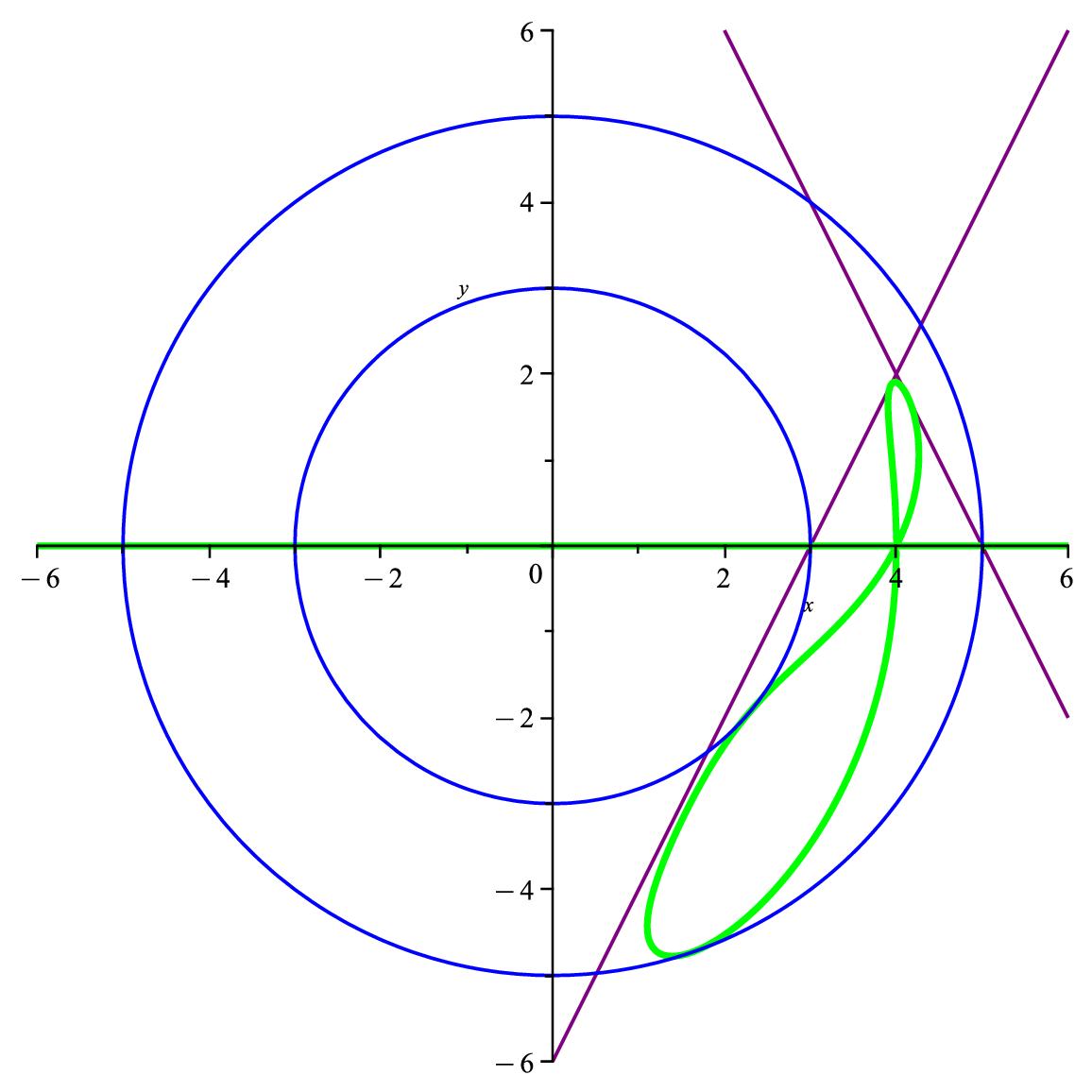}} \hspace*{0.1cm}
 \fbox{\includegraphics[scale=0.33]{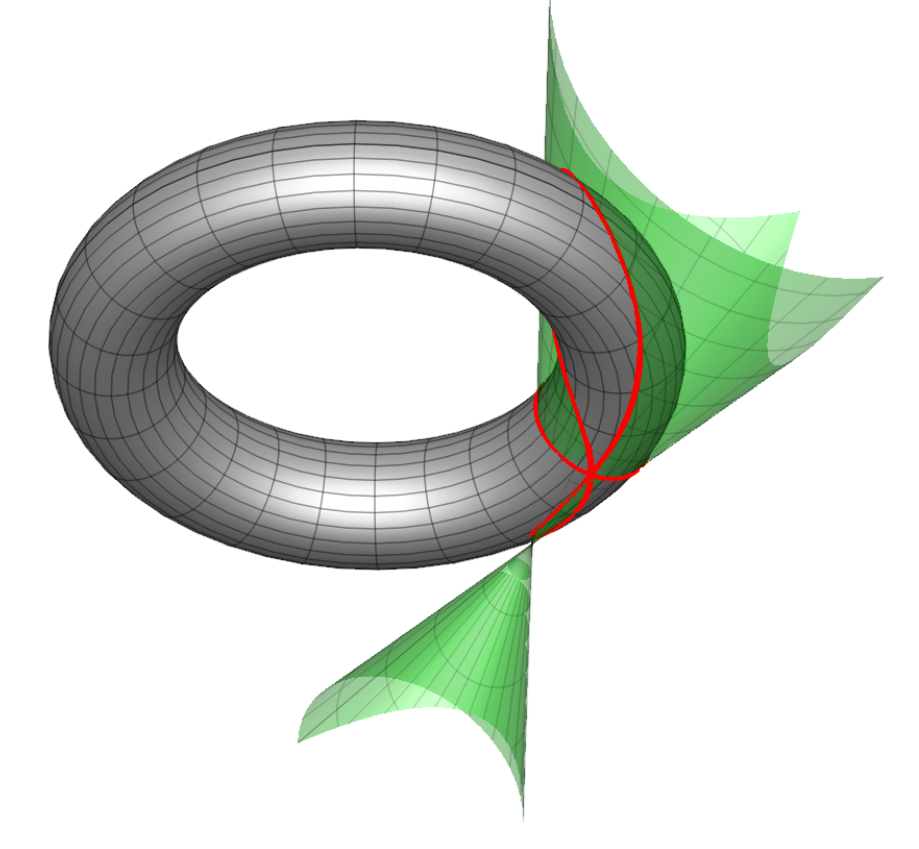}}\;
 \fbox{\includegraphics[scale=0.70]{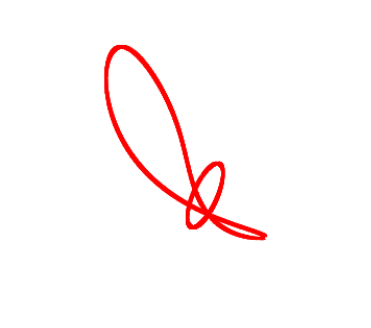}}
\caption{\label{fig_cuentasjorge_arribanoabajosi}
Left: $\Delta_{\cal T}$ (blue), $\Delta_{\cal Q}$ (purple) and the cutcurve (green). Center/Right: Intersection curve (red) for Example \ref{cuentasjorge_arribanoabajosi}. }
\end{figure}
}
\end{examp}

\begin{examp}\label{q10tang}
{\rm 
Consider the torus with $r=2$ and $R=7$, and the hyperboloid
$$ Q(x,y,z)=z^2 - x^2 - y^2  +  25=0 .$$
Then $q_1 \equiv 0$, $ \widetilde{{\bf S}_0} =16(x^2 + y^2 - 4)^2(x^2 + y^2 - 25)^2=4(x^2 + y^2 - 4)^2\Delta_{\cal Q }^2$, the cutcurve is defined by $\Delta_{\cal Q }=x^2 + y^2 - 25=0$ and observe that $ x^2 + y^2 - 25$ is the interior circle of $\Delta_{\cal T}$. Moreover, the intersection, the minimum parallel, is tangential along all the space curve, see Figure \ref{figq10tang}. The hypotheses of Theorem \ref{Singq_1Siemprecero_enCorona} are not verified, and so there are no singularities. This example illustrates Theorem \ref{lma:mult_comp}. 

 \begin{figure}[H]
\centering
 \fbox{\includegraphics[scale=0.45]{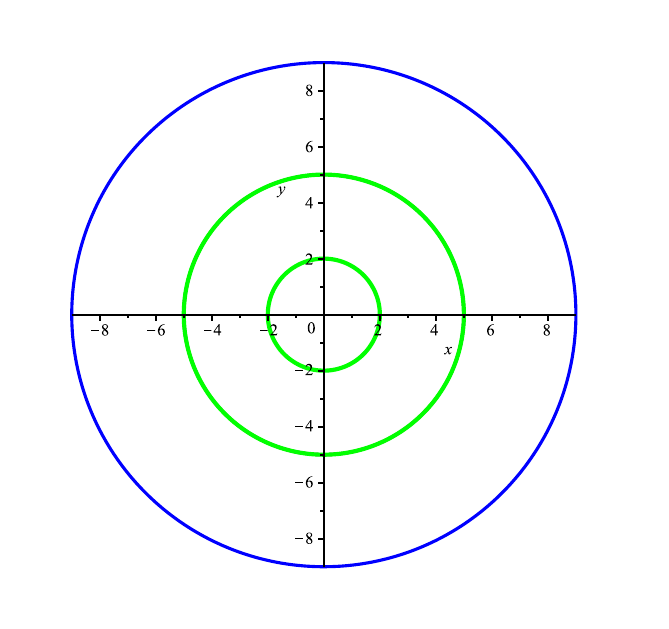}} \hspace*{0.1cm}
 \fbox{\includegraphics[scale=0.30]{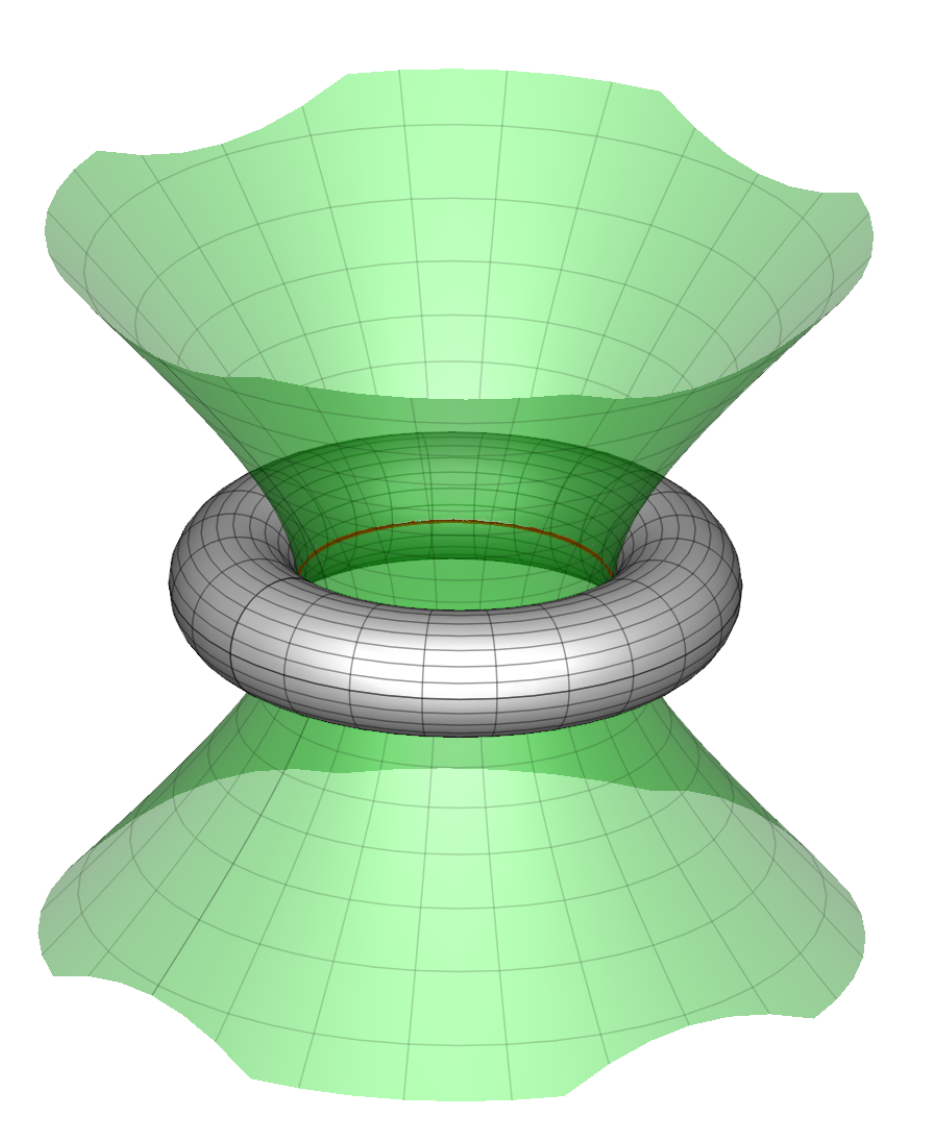}}\;
\caption{\label{figq10tang}
Left: $\Delta_{\cal T}$ (blue) and the cutcurve (green). Right: Intersection curve (red) for Example \ref{q10tang}. }
\end{figure}
}
\end{examp}

\begin{examp}\label{sing_singreg}
 {\rm Consider the torus with $r=5$ and $R=10$, and the  hyperboloid
$$ Q(x,y,z)=  z^{2}-x^{2} +2 y^{2}+\left(x +2 y -13\right) z+4 x y+28 x -60 y -211 .$$
In this case, the resultant is squarefree. The point $(13, 0)$ satisfies $q_1=0$, and so it is singular on the cutcurve. We observe that this point is the projection of  both $(13, 0,-4)$ and $(13, 0,4)$. While the point $(13, 0,4)$ corresponds to a singular point on the intersection curve, the point $(13, 0,-4)$ is regular. 
This example illustrates Theorem \ref{singarribaabajo} and Corollary \ref{singnotwosing}. See Figure \ref{Fsing_singreg}.   
 
\begin{figure}[H]
\centering
 \fbox{\includegraphics[scale=0.25]{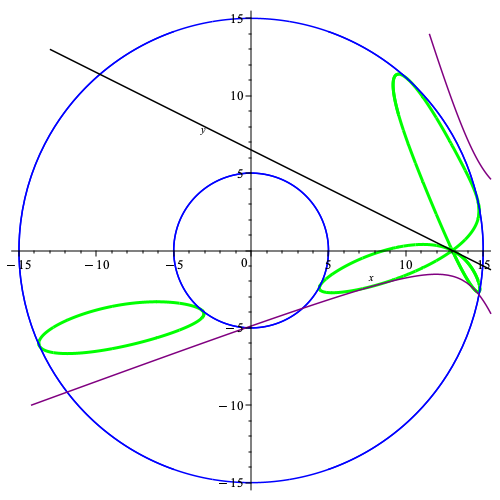}} \hspace*{0.1cm}
 \fbox{\includegraphics[scale=0.35]{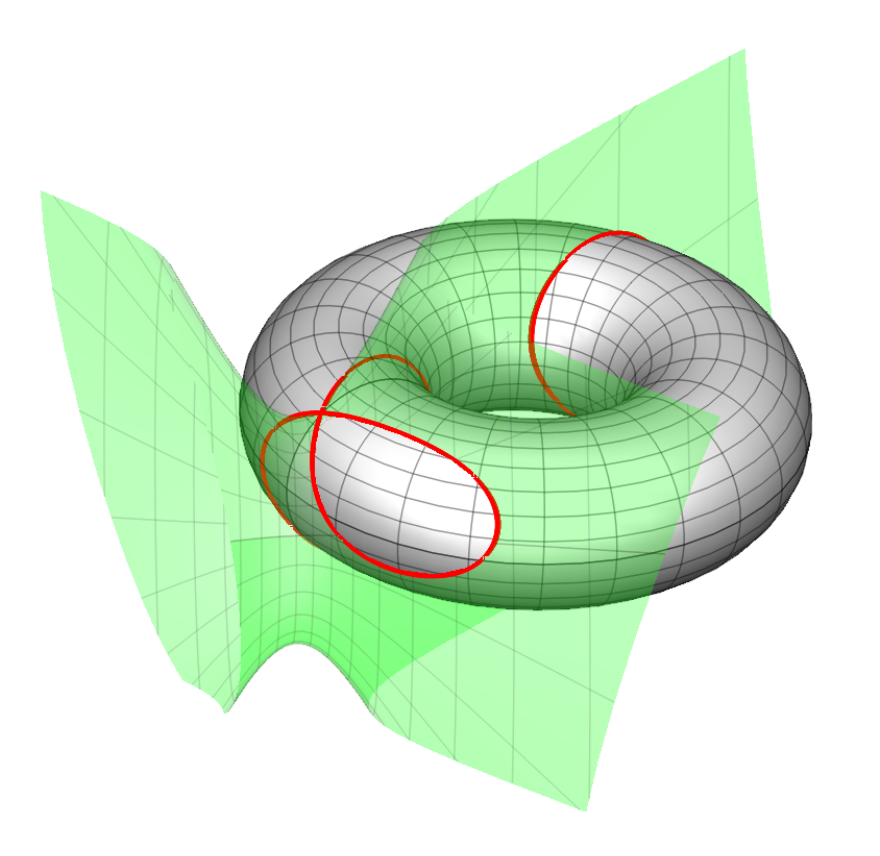}}\;
 \fbox{\includegraphics[scale=0.2]{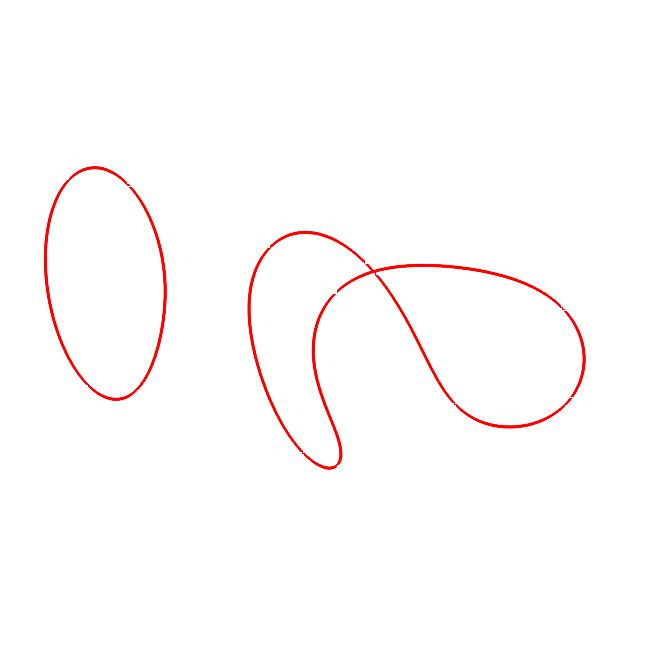}}
\caption{\label{Fsing_singreg}
Left: $\Delta_{\cal T}$ (blue), $\Delta_{\cal Q}$ (purple), the line $q_1$ (black) and the cutcurve (green). Center/Right: Intersection curve (red) for Example \ref{sing_singreg}. }
\end{figure}
}
\end{examp}

\begin{examp}\label{cono}
 {\rm 
Consider the torus with $r=4$ and $R=6$, and the cone
$$  Q(x,y,z)=z^2 - (x - 2)^2-y^2=0$$
Here $q_1\equiv 0$, $ \widetilde{{\bf S}_0}= 16 \left(x^{4}+2 x^{2} y^{2}+y^{4}-4 x^{3}-4 x \,y^{2}-8 x^{2}-12 y^{2}-48 x +144\right)^{2} $ and the cutcurve is defined by the squarefree part of $\widetilde{{\bf S}_0}$. In Figure \ref{cono_gm_mw}, observe that the real plot of $\Delta_{\cal{Q}}$ is just a point, which is the projection of the singular point of the cone. This example illustrates Proposition \ref{singularidadcono} and Theorem \ref{Singq_1Siemprecero_enCorona}.

\begin{figure}[H]
\centering
 \fbox{\includegraphics[scale=0.25]{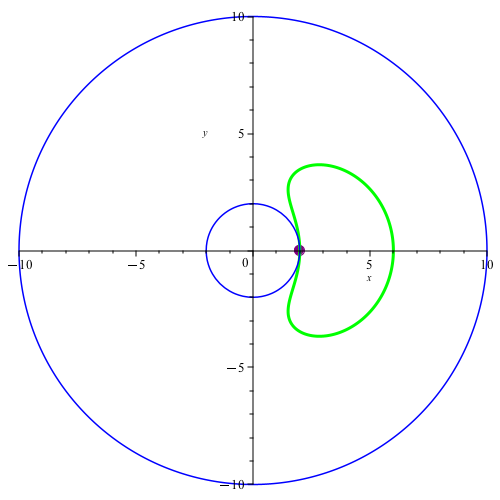}} \hspace*{0.1cm}
 \fbox{\includegraphics[scale=0.3]{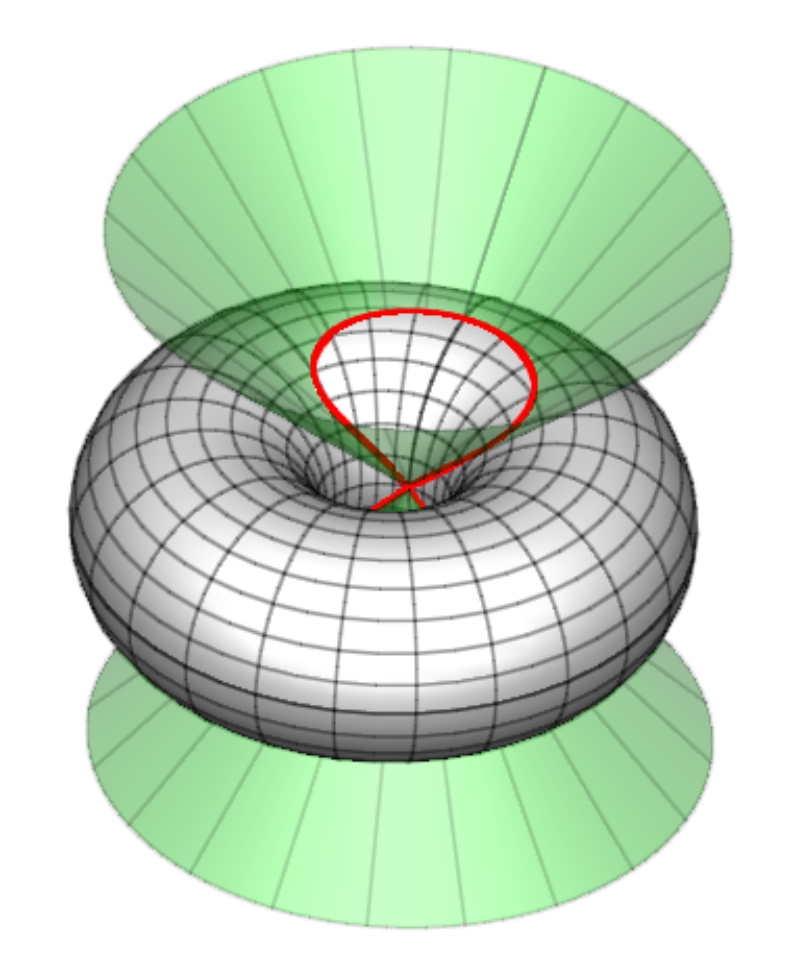}}
\caption{\label{cono_gm_mw}
Left: $\Delta_{\cal T}$ (blue), $\Delta_{\cal Q}$ (a purple point) and the cutcurve (green). Right: Intersection curve (red) for Example \ref{cono}. }
\end{figure}
}
\end{examp}

\section{\bf Acknowledgements. }
Authors are partially supported by Ministerio de Ciencia, Innovaci\'on y Universidades - Agencia Estatal de Investigaci\'on PID2020-113192GB-I00 (Mathematical Visualization: Foundations, Algorithms and Applications).


\bibliographystyle{elsarticle-num}

\end{document}